\documentclass[11pt,twoside, leqno]{article}

\usepackage{amssymb}
\usepackage{amsmath}
\usepackage{amsthm}

\allowdisplaybreaks

\def\rr{{\mathbb R}}
\def\rn{{{\rr}^n}}

\def\rnz{{\rr}^{n+1}_+}
\def\zz{{\mathbb Z}}
\def\nn{{\mathcal N}}

\def\ch{{\mathcal H}}
\def\cc{{\mathbb C}}
\def\cn{{\mathbb N}}
\def\cs{{\mathcal S}}
\def\cx{{\mathcal X}}

\def\cd{{\mathcal D}}

\def\cl{{\mathfrak{L}}}

\def\cm{{\mathcal M}}
\def\car{{\mathcal R}}

\def\ca{{\mathcal A}}

\def\nz{{\nabla}}

\def\fz{\infty}
\def\az{\alpha}
\def\supp{{\mathop\mathrm{\,supp\,}}}
\def\dist{{\mathop\mathrm{\,dist\,}}}
\def\loc{{\mathop\mathrm{\,loc\,}}}

\def\lz{\lambda}
\def\dz{\delta}

\def\ez{\epsilon}

\def\kz{\kappa}
\def\bz{\beta}
\def\ro{\rho}

\def\gz{{\gamma}}
\def\oz{{\omega}}

\def\vz{\varphi}

\def\sz{\sigma}
\def\pa{\partial}
\def\wz{\widetilde}

\def\hs{\hspace{0.3cm}}

\def\ls{\lesssim}

\def\bmo{{\mathop\mathrm{BMO}_{\ro,L}(\cx)}}

\def\bmoo{{\mathop\mathrm{BMO}^M_{\ro,L}(\cx)}}

\def\com{\complement}
\def\r{\right}
\def\lf{\left}
\def\lfr{\lfloor}
\def\rf{\rfloor}
\def\la{\langle}
\def\ra{\rangle}

\newtheorem{thm}{Theorem}[section]
\newtheorem{lem}{Lemma}[section]
\newtheorem{prop}{Proposition}[section]
\newtheorem{rem}{Remark}[section]
\newtheorem{cor}{Corollary}[section]
\newtheorem{defn}{Definition}[section]

\numberwithin{equation}{section}

\begin{document}

\arraycolsep=1pt
\title{{\vspace{-5cm}\small\hfill\bf Commun. Contemp. Math., to appear}\\
\vspace{4cm}\Large\bf Orlicz-Hardy Spaces Associated with
Operators Satisfying Davies-Gaffney Estimates
\footnotetext{\hspace{-0.35cm} 2000 {\it Mathematics
Subject Classification}. Primary 42B35; Secondary 42B30; 42B25.
\endgraf{\it Key words and phrases.} nonnegative self-adjoint operator,
Schr\"odinger operator, Riesz transform, Davies-Gaffney estimate,
Orlicz function, Orlicz-Hardy space, Lusin area function,
maximal function, atom, molecule, dual, BMO.
\endgraf
Dachun Yang is supported by the National
Natural Science Foundation (Grant No. 10871025) of China.
\endgraf $^\ast\,$Corresponding author.}}
\author{Renjin Jiang and Dachun Yang$\,^\ast$ }
\date{ }
\maketitle

\begin{center}
\begin{minipage}{13.5cm}\small
{\noindent{\bf Abstract.} Let ${\mathcal X}$ be a metric space with
doubling measure, $L$ a nonnegative self-adjoint operator in
$L^2({\mathcal X})$ satisfying the Davies-Gaffney estimate, $\omega$
a concave function on $(0,\infty)$ of strictly lower type
$p_\omega\in (0, 1]$ and $\rho(t)={t^{-1}}/\omega^{-1}(t^{-1})$ for
all $t\in (0,\infty).$ The authors introduce the Orlicz-Hardy space
$H_{\omega,L}({\mathcal X})$ via the Lusin area function associated
to the heat semigroup, and the BMO-type space
${\mathop\mathrm{BMO}_{\rho,L}(\mathcal X)}$.  The authors then
establish the duality between $H_{\omega,L}({\mathcal X})$ and
$\mathrm{BMO}_{\rho,L}({\mathcal X})$; as a corollary, the authors
obtain the $\rho$-Carleson measure characterization of the space
${\mathop\mathrm{BMO}_{\rho,L}(\mathcal X)}$. Characterizations of
$H_{\omega,L}({\mathcal X})$, including the atomic and molecular
characterizations and the Lusin area function characterization
associated to the Poisson semigroup, are also presented. Let
${\mathcal X}={\mathbb R}^n$ and $ L=-\Delta+V$ be a Schr\"odinger
operator, where $V\in L^1_{\mathrm{\,loc\,}}({\mathbb R}^n)$ is a
nonnegative potential. As applications, the authors show that the
Riesz transform $\nabla L^{-1/2}$ is bounded from
$H_{\omega,L}({{\mathbb R}^n})$ to $L(\omega)$; moreover, if there
exist $q_1,\,q_2\in (0,\infty)$ such that $q_1<1<q_2$ and
{\normalsize$[\omega(t^{q_2})]^{q_1}$} is a convex function on
$(0,\infty)$, then several characterizations of the Orlicz-Hardy
space $H_{\omega,L}({{\mathbb R}^n})$, in terms of the Lusin-area
functions, the non-tangential maximal functions, the radial maximal
functions, the atoms and the molecules, are obtained. All these
results are new even when $\omega(t)=t^p$ for all $t\in (0,\infty)$
and $p\in (0,1)$.}
\end{minipage}
\end{center}

\section{Introduction\label{s1}}

\hskip\parindent The theory of Hardy spaces $H^p$ in various
settings plays an important role in analysis and partial
differential equations. However, the classical theory of Hardy
spaces on $\rn$ is intimately connected with the Laplacian operator.
In recent years, the study of Hardy spaces and BMO spaces associated
with different operators inspired great interests; see, for example,
\cite{adm,amr,ar,dy1,dy2,dz1,dz2,hm1,hlmmy,jy,ya1} and their
references. In \cite{adm}, Auscher, Duong and McIntosh studied the
Hardy space $H_L^1(\rn)$ associated to an operator $L$ whose heat
kernel satisfies a pointwise Poisson upper bound. Later, in
\cite{dy1,dy2}, Duong and Yan introduced the BMO-type space
$\mathrm{BMO}_L(\rn)$ associated to such an $L$ and established the
duality between $H_L^1(\rn)$ and $\mathrm{BMO}_{L^\ast}(\rn)$, where
$L^\ast$ denotes the adjoint operator of $L$ in $L^2(\rn)$. Yan
\cite{ya1} further generalized these results to the Hardy space
$H^p_L(\rn)$ with $p\in(0, 1]$ close to 1 and its dual space. Very
recently, Auscher, McIntosh and Russ \cite{amr} treated the Hardy
space $H^1$ associated to the Hodge Laplacian on a Riemannian
manifold with doubling measure; Hofmann and Mayboroda \cite{hm1}
introduced the Hardy space $H^1_L(\rn)$ and its dual space adapted
to a second order divergence form elliptic operator $L$ on $\rn$
with complex coefficients. Notice that these operators may not have
the pointwise heat kernel bounds. Furthermore, Hofmann et al
\cite{hlmmy} studied the Hardy space $H^1_L(\cx)$ on a metric
measured space $\cx$ adapted to $L$, which is nonnegative
self-adjoint, and satisfies the so-called Davies-Gaffney estimate.

On the other hand, as another generalization of $L^p(\rn)$, the
Orlicz space was introduced by Birnbaum-Orlicz in \cite{bo} and
Orlicz in \cite{o32}. Since then, the theory of the Orlicz spaces
themselves has been well developed and the spaces have been widely
used in probability, statistics, potential theory, partial
differential equations, as well as harmonic analysis and some other
fields of analysis; see, for example, \cite{rr91,rr00}. Moreover,
the Orlicz-Hardy spaces are also good substitutes of the Orlicz
spaces in dealing with many problems of analysis. In particular,
Str\"omberg \cite{s79}, Janson \cite{ja} and Viviani \cite{v}
studied Orlicz-Hardy spaces and their dual spaces.

Recall that the Orlicz-Hardy spaces associated operators on $\rn$
have been studied in \cite{jyz, jy}. In \cite{jyz}, the heat kernel
is assumed to enjoy a pointwise Poisson type upper bound; while in
\cite{jy}, $L$ is a second order divergence form elliptic operator
on $\rn$ with complex coefficients. Motivated by
\cite{hm1,hlmmy,ja,v}, in this paper, we study the Orlicz-Hardy
space $H_{\oz,L}(\cx)$ and its dual space associated with a
nonnegative self-adjoint operator $L$ on a metric measured space
$\cx$.

Let $\cx$ be a metric space with doubling measure and $L$ a
nonnegative self-adjoint operator in $L^2(\cx)$ satisfying the
Davies-Gaffney estimate. Let $\oz$ on $(0,\fz)$ be a concave
function of strictly lower type $p_\oz\in (0, 1]$ and
$\rho(t)={t^{-1}}/\oz^{-1}(t^{-1})$ for all $t\in (0,\fz).$ A
typical example of such Orlicz functions is $\oz(t)=t^p$ for all
$t\in (0,\fz)$ and $p\in (0,1]$. To develop a real-variable theory
of the Orlicz-Hardy space $H_{\oz,L}(\cx)$, the key step is to
establish an atomic (molecular) characterization of these spaces. To
this end, we inherit a method used in \cite{amr,jy}. We first
establish the atomic decomposition of the tent space $T_{\oz}(\cx)$,
whose proof implies that if $F\in T_{\oz}(\cx)\cap T_2^2(\cx)$, then
the atomic decomposition of $F$ holds in both $T_{\oz}(\cx)$ and
$T_2^2(\cx)$. Then by the fact that the operator $\pi_{\Psi,L}$ (see
\eqref{4.6}) is bounded from $T_2^2(\cx)$ to $L^2(\cx)$, we further
obtain the $L^2(\cx)$-convergence of the corresponding atomic
decomposition for functions in $H_{\oz,L}(\cx)\cap L^2(\cx)$, since
for all $f\in H_{\oz,L}(\cx)\cap L^2(\cx)$, $t^2Le^{-t^2L}f\in
T_2^2(\cx)\cap T_{\oz}(\cx)$. This technique plays a fundamental
role in the whole paper.

With the help of the atomic decomposition, we establish the dual
relation between the spaces $H_{\oz,L}(\cx)$ and
$\mathrm{BMO}_{\ro,L}(\cx)$. As a corollary, we obtain the
$\ro$-Carleson measure characterization of the space $\bmo$. Having
at hand the duality relation, we then obtain the atomic and
molecular characterizations of the space $H_{\oz,L}(\cx)$. We also
introduce the Orlicz-Hardy space $H_{\oz,\cs_P}(\cx)$ via the Lusin
area function associated to the Poisson semigroup. With the atomic
characterization of $H_{\oz,L}(\cx)$, we finally show that the
spaces $H_{\oz,\cs_P}(\cx)$ and $H_{\oz,L}(\cx)$ coincide with
equivalent norms. Let $\cx=\rn$ and $L=-\Delta+V$, where $V\in
L^1_\loc(\rn)$ is a nonnegative potential.  As applications, we show
that the Riesz transform $\nabla L^{-1/2}$ is bounded from
$H_{\oz,L}(\rn)$ to $L(\oz)$; moreover, if there exist $q_1,\,q_2\in
(0,\fz)$ such that $q_1<1<q_2$ and $[\oz(t^{q_2})]^{q_1}$ is a
convex function on $(0,\fz)$, then we obtain several
characterizations of $H_{\oz,L}(\rn)$, in terms of the Lusin-area
functions, the non-tangential maximal functions, the radial maximal
functions, the atoms and the molecules. Notice that here, the
potential $V$ is not assumed to satisfy the reverse H\"older
inequality.

Notice that the assumption that $L$ is nonnegative self-adjoint
enables us to obtain an atomic characterization of $H_{\oz,L}(\cx)$.
The method used in the proof of atomic characterization depends on the
finite speed propagation property for solutions of the corresponding
wave equation of $L$ and hence the self-adjointness of $L$. Without
self-adjointness, as in \cite{adm,dy2,hm1,jyz,jy,ya1}, where $L$
satisfies $H_\fz$-functional calculus and the heat kernel
generated by $L$ satisfies a pointwise Poisson type upper
bound or the Davies-Gaffney estimate, a corresponding
(Orlicz-)Hardy space theory
with the molecular ({\it not} atomic) characterization was also
established in \cite{adm,dy2,hm1,jyz,jy,ya1}.

Precisely, this paper is organized as follows. In Section 2, we
first recall some definitions and notation concerning metric
measured spaces $\cx$, then describe some basic assumptions on the
operator $L$ and the Orlicz function $\oz$ and present some
properties of the operator $L$ and Orlicz functions considered in
this paper.

In Section 3, we first recall some notions about tent spaces and then
study the tent space $T_\oz(\cx)$ associated to the Orlicz
function $\oz$. The main result of this section is that we characterize the
tent space $T_\oz(\cx)$ by the atoms; see Theorem \ref{t3.1} below.
As a byproduct, we see that if $f\in T_\oz(\cx)\cap T_2^2(\cx)$, then the atomic
decomposition holds in both $T_\oz(\cx)$ and $T_2^2(\cx)$, which plays an important
role in the remaining part of this paper; see Corollary \ref{c3.1} below.

In Section 4,  we first introduce the Orlicz-Hardy space
$H_{\oz,L}(\cx)$ and prove that the operator $\pi_{\Psi,L}$ (see
\eqref{4.6} below) maps the tent space $T_\oz(\cx)$ continuously
into $H_{\oz,L}(\cx)$ (see Proposition \ref{p4.2} below). By this
and the atomic decomposition of $T_\oz(\cx)$, we obtain that for
each $f\in H_{\oz,L}(\cx)$, there is an atomic decomposition of $f$
holding in $H_{\oz,L}(\cx)$ (see Proposition \ref{p4.3} below). We
should point out that to obtain the atomic decomposition of
$H_{\oz,L}(\cx)$, we borrow a key idea from \cite{hlmmy}, namely,
for a nonnegative self-adjoint operator $L$ in $L^2(\cx)$, then $L$
satisfies the Davies-Gaffney estimate if and only if it has the
finite speed propagation property; see \cite{hlmmy} (or Lemma
\ref{l2.2} below). Via this atomic decomposition of
$H_{\oz,L}(\cx)$, we further obtain the duality between
$H_{\oz,L}(\cx)$ and $\bmo$ (see Theorem \ref{t4.1} below). As an
application of this duality, we establish a $\ro$-Carleson measure
characterization of the space $\bmo$; see Theorem \ref{t4.2} below.
We point out that if $\cx=\rn$, $L=-\Delta\equiv -\sum_{i=1}^n
\frac{\pa^2}{\pa x_i^2}$ and $\oz$ is as above with $p_\oz \in
(n/(n+1),1]$, then the Orlicz-Hardy space $H_{\oz,L}(\rn)$ in this
case coincides with the Orlicz-Hardy space in \cite{jyz} and it was
proved there that $H_{\oz,L}(\rn)=H_\oz(\rn)$;  see \cite{ja,v} for
the definition of $H_\oz(\rn)$.

In Section 5, by Proposition \ref{p4.3} and Theorem \ref{t4.1}, we
establish the equivalence of $H_{\oz,L}(\cx)$ and the atomic (resp.
molecular) Orlicz-Hardy $H^{M}_{\oz,{\rm at}}(\cx)$ (resp.
$H^{M,\ez}_{\oz,{\rm mol}}(\cx)$); see Theorem \ref{t5.1} below. We
notice that the series in $H^{M}_{\oz,{\rm at}}(\cx)$ (resp.
$H^{M,\ez}_{\oz,{\rm mol}}(\cx)$) is defined to converge in the norm
of $(\bmo)^\ast$; while in Corollary \ref{c4.1} below, the atomic
decomposition holds in $H_{\oz,L}(\cx)$. Applying the atomic
characterization, we further characterize the Orlicz-Hardy space
$H_{\oz,L}(\cx)$ in terms of the Lusin area function associated to
the Poisson semigroup; see Theorem \ref{5.2} below.

As applications, in Section 6, we study the Hardy spaces
$H_{\oz,L}(\rn)$ associated to the Schr\"odinger operator
$L=-\Delta+V$, where $V\in L^1_\loc(\rn)$ is a nonnegative
potential. We characterize $H_{\oz,L}(\rn)$ in terms of the
Lusin-area functions, the atoms and the molecules; see Theorem
\ref{t6.1} below. Moreover, we show that the Riesz transform $\nabla
L^{-1/2}$ is bounded from $H_{\oz,L}(\rn)$ to $L(\oz)$ and from
$H_{\oz,L}(\rn)$ to the classical Orlicz-Hardy space $H_{\oz}(\rn)$,
if $p_\oz\in (\frac n{n+1},1]$; see Theorems \ref{t6.2} and
\ref{t6.3} below. If there exist $q_1,\,q_2\in (0,\fz)$ such that
$q_1<1<q_2$ and $[\oz(t^{q_2})]^{q_1}$ is a convex function on $
(0,\fz)$, then we obtain several characterizations of
$H_{\oz,L}(\rn)$, in terms of the non-tangential maximal functions
and the radial maximal functions; see Theorem \ref{t6.4} below.
Denote $H_{\oz,L}(\rn)$ by $H_L^p(\rn)$, when $p\in (0,1]$ and
$\oz(t)=t^p$ for all $t\in (0,\fz)$. We remark that the boundedness
of $\nabla L^{-1/2}$ from $H_L^1({{\mathbb R}^n})$ to the classical
Hardy space $H^1({{\mathbb R}^n})$ was established in \cite{hlmmy}.
Moreover, if $n=1$ and $p=1$, the Hardy space $H_L^1(\rn)$ coincides
with the Hardy space introduced by Czaja and Zienkiewicz in
\cite{cz}; if $L=-\Delta+V$ and $V$ belongs to the reverse H\"older
class $\ch_q(\rn)$ for some $q\ge n/2$ with $n\ge 3$, then the Hardy
space $H_L^p(\rn)$ when $p\in (n/(n+1),1]$ coincides with the Hardy
space introduced by Dziuba\'nski and Zienkiewicz \cite{dz1,dz2}.

Finally, we make some conventions. Throughout the paper, we denote
by $C$ a positive constant which is independent of the main
parameters, but it may vary from line to line. The symbol $X \ls Y$
means that there exists a positive constant $C$ such that $X \le
CY$; the symbol $\lfr\,\az\,\rf$ for $\az\in\rr$ denotes the maximal
integer no greater than $\az$; $B(z_B,\,r_B)$ denotes an open ball
with center $z_B$ and radius $r_B$ and $CB(z_B,\,r_B)\equiv
B(z_B,\,Cr_B).$ Set $\cn\equiv\{1,2,\cdots\}$ and
$\zz_+\equiv\cn\cup\{0\}.$ For any subset $E$ of $\cx$, we denote by
$E^\com$ the set $\cx\setminus E.$ We also use $C(\gz,\bz, \cdots)$
to denote a positive constant depending on the indicated parameters
$\gz,\bz,\cdots$.

\section{Preliminaries\label{s2}}

\hskip\parindent In this section, we first recall some notions and
notation on metric measured spaces and then describe some basic
assumptions on the operator $L$ studied in this paper; finally we
present some basic properties on Orlicz functions and also describe
some basic assumptions of them.

\subsection{Metric measured spaces\label{s2.1}}

\hskip\parindent Throughout the whole paper,
we let $\cx$ be a set, $d$ a metric on
$\cx$ and $\mu$ a nonnegative Borel regular measure
on $\cx$. Moreover, we assume that
there exists a constant $C_1\ge1$ such that for all
$x\in\cx$ and $r>0$,
\begin{equation}\label{2.1}
V(x,2r)\le C_1V(x,r)<\fz,
\end{equation}
where $B(x,r)\equiv \{y\in \cx:\,d(x,y)<r\}$ and
\begin{equation}\label{2.2}
V(x,r)\equiv \mu(B(x,r)).
\end{equation}

Observe that if $d$ is a quasi-metric, then $(\cx,d,\mu)$
is called a space of homogeneous type in
the sense of Coifman and Weiss \cite{cw}.

Notice that the doubling property \eqref{2.1} implies the following strong
homogeneity property that
\begin{equation}\label{2.3}
  V(x,\lz r)\le C\lz^n V(x,r)
\end{equation}
for some positive constants $C$ and $n$ uniformly
for all $\lz\ge1$, $x\in\cx$ and $r>0$.
The parameter $n$ measures the dimension of the space $\cx$
in some sense. There also exist constants $C>0$ and $0\le N\le n$ such that
\begin{equation}\label{2.4}
  V(x,r)\le C\lf(1+\frac{d(x,y)}{r}\r)^NV(y,r)
\end{equation}
uniformly for all $x,\, y\in\cx$ and $r>0$. Indeed, the property
\eqref{2.4} with $N=n$ is a simple corollary of the strong
homogeneity property (2.3). In the cases of Euclidean spaces, Lie
groups of polynomial growth and more generally in Ahlfors regular
spaces, $N$ can be chosen to be $0$.

In what follows, for each ball $B\subset \cx$, we set
\begin{equation}\label{2.5}
U_0(B)\equiv B \ \mathrm{and } \ U_j(B)\equiv 2^jB\backslash 2^{j-1}B \
\mathrm {for}\ j\in\cn.
\end{equation}

\subsection{Assumptions on operators $L$}

\hskip\parindent Throughout the whole paper, as in \cite{hlmmy},
we always suppose that the considered operators $L$ satisfy the
following assumptions.

\begin{proof}[\bf Assumption (A)]\rm  The operator $L$ is
a nonnegative self-adjoint
operator in $L^2(\cx)$.
\end{proof}

\begin{proof}[\bf Assumption (B)]\rm
The semigroup $\{e^{-tL}\}_{t>0}$ generated by $L$ is analytic on
$L^2(\cx)$ and satisfies the Davies-Gaffney estimates, namely, there
exist positive constants $C_2$ and $C_3$ such that for all closed
sets $E$ and $F$ in $\cx$, $t\in (0,\fz)$ and $f\in L^2(E)$,
\begin{equation}\label{2.6}
\|e^{-tL}f\|_{L^2(F)}\le C_2 \exp\bigg\{-\frac{\dist(E,F)^2}{C_3t}\bigg\}
\|f\|_{L^2(E)},
\end{equation}
where and in what follows, $\dist(E,F)\equiv \inf_{x\in E,\,y\in F}d(x,y)$
and $L^2(E)$ is the set of all $\mu$-measurable functions
on $E$ such that $\|f\|_{L^2(E)}=\{\int_E|f(x)|^2\,d\mu(x)\}^{1/2}<\fz$.
\end{proof}

Examples of operators satisfying Assumptions (A) and (B) include
second order elliptic self-adjoint operators in divergence form on
$\rn$, degenerate Schr\"odinger operators with nonnegative
potential, Schr\"odinger operators with nonnegative potential and
magnetic field and Laplace-Beltrami operators on all complete
Riemannian manifolds; see for example, \cite{da,ga,sh,si}.

By Assumptions (A) and (B), we have the following
results which were established in \cite{hlmmy}.

\begin{lem}\label{l2.1}
Let $L$ satisfy Assumptions (A) and (B). Then for any fixed
$k\in\zz_+$ (resp. $j,\,k\in\zz_+$ with $j\le k$), the family
$\{(t^2L)^ke^{-t^2L}\}_{t>0}$ (resp.
$\{(t^2L)^j(I+t^2L)^{-k}\}_{t>0}$) of operators satisfies the
Davies-Gaffney estimates \eqref{2.6} with positive constants
$C_2,\,C_3$ depending on $n,\,k$ (resp. $n,\,j,\,k$) only.
\end{lem}

In what follows, for any operator $T$, let $K_T$ denote its integral
kernel when this kernel exists. By \cite[Proposition 3.4]{hlmmy}, we
know that if $L$ satisfies Assumptions (A) and (B), and
$T=\cos(t\sqrt L)$, then there exists a positive constant $C_4$ such
that
\begin{equation}\label{2.7}
\supp K_T\subset \cd_t\equiv \lf\{(x,y)\in \cx\times\cx:\,d(x,y)\le C_4t\r\}.
\end{equation}
This observation plays a key role in obtaining the atomic
characterization of the Orlicz-Hardy space $H_{\oz, L}(\cx)$; see
\cite{hlmmy} and Proposition \ref{p4.3} below.

\begin{lem}\label{l2.2}
Suppose that the operator $L$ satisfies Assumptions (A) and (B). Let
$\vz\in C_0^\fz(\rr)$ be even and $\supp \vz\subset (-C_4^{-1},C_4^{-1})$, where
$C_4$ is as in \eqref{2.7}. Let $\Phi$ denote the Fourier transform of $\vz$. Then
for every $\kz\in\zz_+$ and $t>0$, the kernel $K_{(t^2L)^\kz \Phi(t\sqrt L)}$ of
$(t^2L)^\kz \Phi(t\sqrt L)$ satisfies that
$\supp K_{(t^2L)^\kz \Phi(t\sqrt L)}\subset \lf\{(x,y)\in
\cx\times\cx:\,d(x,y)\le t\r\}.$
\end{lem}

The following estimate is often used in this paper.
Let $\cl_{\cc\to\cc}$ denote the set of all measurable
functions from $\cc$ to $\cc$. For $\dz>0$, define
$$F(\dz)\equiv \lf\{\psi\in\cl_{\cc\to\cc}: \mathrm{there\  exists} \ C>0 \
\mathrm{ such\  that\ for\ all\ }z\in\cc,
\,|\psi(z)|\le C\frac{|z|^\dz}{1+|z|^{2\dz}}\r\}.$$
Then for any non-zero function $\psi\in F(\dz)$, we have
$\int_0^\fz|\psi(t)|^2\frac{\,dt}{t}<\fz$. It was proved in \cite{hlmmy} that
for all $f\in L^2(\cx)$,
\begin{equation}\label{2.8}
\int_0^\fz\|\psi(t\sqrt L)f\|_{L^2(\cx)}^2\frac{\,dt}{t}
=\int_0^\fz|\psi(t)|^2\frac{\,dt}{t}
\|f\|^2_{L^2(\cx)}.
\end{equation}

\subsection{Orlicz functions \label{s2.3}}

\hskip\parindent Let $\omega$ be a positive function defined on
$\rr_+\equiv(0,\,\fz).$ The function $\omega$ is said to be of upper
(resp. lower) type $p$  for some $p\in[0,\,\fz)$, if there exists a
positive constant $C$ such that for all $t\geq 1$ (resp. $t\in (0,
1]$) and $s\in (0,\fz)$,
\begin{equation}\label{2.9}
\omega(st)\le Ct^p \omega(s).
\end{equation}

Obviously, if $\oz$ is of lower type $p$ for some $p>0$, then
$\lim_{t\to0+}\oz(t)=0.$ So for the sake of convenience, if it is
necessary, we may assume that $\oz(0)=0.$ If $\oz$ is of both upper
type ${p_1}$ and lower type $p_0$, then $\oz$ is said to be of type
$(p_0,{p_1}).$ Let
\begin{equation*}
p_\oz^+\equiv\inf\{ p>0:\ \mathrm{there\ exists} \ C>0 \ \mathrm{such \ that }
\ \eqref{2.9} \ \mathrm{holds\ for\ all}\ t\in[1,\fz),\ s\in (0,\fz)\},
\end{equation*}
and
\begin{equation*}
p_\oz^-\equiv\sup\{ p>0:\ \mathrm{there\ exists} \ C>0 \ \mathrm{such \ that }
\ \eqref{2.9} \ \mathrm{holds\ for\ all}\
t\in(0,1],\ s\in (0,\fz)\}.
\end{equation*}
The function $\oz$ is said to be of strictly lower type $p$ if for all $t\in(0,1)$
and $s\in (0,\fz)$, $\omega(st)\le t^p \omega(s),$ and we define
\begin{equation*}
p_\oz\equiv\sup\{ p>0: \omega(st)\le t^p \omega(s) \ \mathrm{holds\ for\
all}\ s\in (0,\fz)\ \mathrm{and}\ t\in(0,1)\}.
\end{equation*} It is easy to see that $p_\oz\le p_\oz^{-}\le{p_\oz^+}$ for all $\oz.$
In what follows, $p_\oz$, $p_\oz^-$ and ${p_\oz^+}$ are called the
strictly critical lower type index, the critical lower
type index and the critical upper type index of $\oz$, respectively.

\begin{rem}\label{r2.1}\rm
We claim that if $p_\oz$ is defined as above, then $\oz$ is also of
strictly lower type $p_\oz$. In other words, $p_\oz$ is attainable.
In fact, if this is not the case, then there exist some $s\in
(0,\fz)$ and $t\in (0,1)$ such that $\oz(st)>t^{p_\oz}\oz(s)$. Hence
there exists $\ez\in(0,p_\oz)$ small enough such that
$\oz(st)>t^{p_\oz-\ez }\oz(s)$, which is contrary to the definition
of $p_\oz$. Thus, $\oz$ is of strictly lower type $p_\oz$.
\end{rem}

Throughout the whole paper, we always assume that $\oz$ satisfies
the following assumption.

\begin{proof}[\bf Assumption (C)] Let $\oz$ be a positive function
defined on $\rr_+$, which is of strictly lower type and its
strictly lower type index $p_\oz\in (0,1]$. Also assume
that $\oz$ is continuous, strictly
increasing and concave.
\end{proof}

Notice that if $\oz$ satisfies Assumption (C), then $\oz(0)=0$
and $\oz$ is obviously of upper type 1. Since $\oz$ is concave,
it is subadditive. In fact, let $0<s<t$, then
$$\oz(s+t)\le \frac{s+t}{t}\oz(t)\le \oz(t)
+\frac{s}{t}\frac{t}{s}\oz(s)=\oz(s)+\oz(t).$$ For any concave
function $\oz$ of strictly lower type $p$, if we set
$\wz\oz(t)\equiv\int_0^t\frac{\oz(s)}{s}\,ds$ for $t\in [0,\fz)$,
then by \cite[Proposition 3.1]{v}, $\wz\oz$ is equivalent to $\oz$,
namely, there exists a positive constant $C$ such that
$C^{-1}\oz(t)\le \wz\oz(t)\le C\oz(t)$ for all $t\in [0,\fz)$;
moreover, $\wz\oz$ is strictly increasing, concave, subadditive and
continuous function of strictly lower type $p.$ Since all our
results are invariant on equivalent functions, we always assume that
$\oz$ satisfies Assumption (C); otherwise, we may replace $\oz$ by
$\wz\oz.$

\begin{proof}[\bf Convention]\rm From Assumption (C), it follows that
$0<p_\oz\le p_\oz^{-}\le {p_\oz^+}\le 1.$ In what follows, if
\eqref{2.9} holds for ${p_\oz^+}$ with $t\in [1,\fz)$, then we
choose ${\wz p_\oz}\equiv{p_\oz^+}$; otherwise  $p_\oz^+<1$ and we
choose ${\wz p_\oz}\in (p_\oz^+,1)$ to be close enough to $p_\oz^+$,
the meaning will be made clear in the context.
\end{proof}

For example, if $\oz(t)=t^p$ with $p\in (0, 1]$
for all $t\in (0,\fz)$, then $p_\oz={p_\oz^+}={\wz p_\oz}=p$;
if $\oz(t)=t^{1/2}\ln(e^4+t)$ for all $t\in (0,\fz)$,
then $p_\oz={p_\oz^+}=1/2$, but $1/2<{\wz p_\oz}<1$.

Let $\oz$ satisfy Assumption (C). A measurable function $f$ on
$\cx$ is said to be in the space $L(\oz)$ if
$\int_{\cx}\oz(|f(x)|)\,d\mu(x)< \fz.$
Moreover, for any $f\in L(\oz)$, define
$$\|f\|_{L(\oz)}\equiv\inf\lf\{\lz>0:\ \int_{\cx}\oz\lf(\frac{|f(x)|}
{\lz}\r)\,d\mu(x)\le 1\r\}.$$

Since $\oz$ is strictly increasing, we define the function
$\ro(t)$ on $\rr_+$ by
\begin{equation}\label{2.10}\ro(t)\equiv\frac{t^{-1}}{\oz^{-1}(t^{-1})}
\end{equation}
for all $t\in (0,\fz)$, where $\oz^{-1}$ is the inverse function of
$\oz.$ Then the types of $\oz$ and $\rho$ have the following
relation; see \cite{v} for a proof.

\begin{prop}\label{p2.1}
Let $0<p_0\le {p_1}\le1$ and $w$ be an increasing function. Then $\oz$
is of type $(p_0,\,{p_1})$ if and only if $\ro$ is of type
$(p_1^{-1}-1,\, p_0^{-1}-1).$
\end{prop}

\section{Tent spaces associated to Orlicz functions\label{s3}}

\hskip\parindent In this section, we study the tent spaces associated
to Orlicz functions $\oz$ satisfying Assumption (C).
 We first recall some notions.

For any $\nu>0$ and $x\in\cx$, let
$\Gamma_\nu(x)\equiv\{(y,t)\in\cx\times(0,\fz):\,d(x,y)<\nu t\}$
denote the  cone of aperture $\nu$ with vertex $x\in\cx$. For any
closed set $F$ of $\cx$, denote by $\car_\nu{F}$ the union of all
cones with vertices in $F$, namely, $\car_\nu{F}\equiv\bigcup_{x\in
F}\Gamma_\nu(x)$; and for any open set $O$ in $\cx$, denote the tent
over $O$ by $T_\nu(O)$, which is defined as
$T_\nu(O)\equiv[\car_\nu(O^\com)]^\com.$ It is easy to see that
$T_\nu(O)=\{(x,t)\in\cx\times(0,\fz):\,d(x,O^\com)\ge \nu t\}.$
In what follows, we denote $\car_1(F)$, $\Gamma_1(x)$ and $T_1(O)$ simply by
$\car(F)$, $\Gamma(x)$ and $\widehat O$, respectively.

For all measurable function $g$ on $\cx\times (0,\fz)$ and
$x\in\cx$, define
\begin{equation*}
\ca_\nu (g)(x)\equiv
\lf(\int_{\Gamma_\nu(x)}|g(y,t)|^2\frac{\,d\mu(y)}{V(x,t)}\frac{\,dt}{t}\r)^{1/2}
\end{equation*}
and denote $\ca_1 (g)$ simply by $\ca(g)$.

If $\cx=\rn$, Coifman, Meyer and Stein \cite{cms} introduced the
tent space $T_2^p({\rr}^{n+1}_+)$ for $p\in(0,\fz)$. The tent spaces
$T_2^p(\cx)$ on spaces of homogenous type were studied by Russ
\cite{ru}. Recall that a measurable function $g$ is said to belong
to the space $T_2^p(\cx)$ with $p\in (0,\fz)$, if
$\|g\|_{T_2^p(\cx)}\equiv \|\ca(g)\|_{L^p(\cx)}<\fz$. On the other
hand, Harboure, Salinas and Viviani \cite{hsv} introduced the tent
space $T_\oz({\rr}^{n+1}_+)$ associated to the function $\oz$.

In what follows, we denote by $T_\oz(\cx)$ the space of all
measurable function $g$ on $\cx\times (0,\fz)$ such that $\ca(g)\in
L(\oz)$, and for any $g\in T_\oz(\cx)$, define its norm by
$$\|g\|_{T_\oz(\cx)}\equiv \|\ca(g)\|_{L(\oz)}=\inf\lf\{\lz>0:\ \int_{\cx} \oz
\lf(\frac{\ca(g)(x)}{\lz}\r) \,d\mu(x)\le 1\r\}.$$

A function $a$ on $\cx\times (0,\fz)$ is called a $T_{\oz}(\cx)$-atom if

(i) there exists a ball $B\subset
\cx$ such that $\supp a\subset \widehat{B};$

(ii) $\int_{\widehat{B}}|a(x,t)|^2\frac{\,d\mu(x)\,dt}{t}\le
[V(B)]^{-1}[\ro(V(B))]^{-2}.$

Since $\oz$ is concave, it is easy to see that for all
$T_{\oz}(\cx)$-atom $a$, we have $\|a\|_{T_{\oz}(\cx)}\le 1.$ In
fact, since $\oz^{-1}$ is convex, by the Jensen inequality and the
H\"older inequality, we have
\begin{eqnarray*}
\oz^{-1}\lf(\frac{\int_B\oz(\ca(a)(x))\,d\mu(x)}{V(B)}\r)
&&\le \frac{1}{V(B)}\int_B\ca(a)(x)\,d\mu(x)\le
\frac{\|a\|_{T_2^2(\cx)}}{[V(B)]^{1/2}}\le \frac{1}{V(B)\ro(V(B))},
\end{eqnarray*}
which implies that
$$ \int_B\oz(\ca(a)(x))\,d\mu(x)\le V(B)
\oz\lf(\frac{1}{V(B)\ro(V(B))}\r)=1.$$
Thus, the claim holds.

For functions in the space $T_\oz(\cx)$, we have the following
atomic decomposition. The proof of Theorem \ref{t3.1} is similarly
to those of \cite[Theorem 1]{cms}, \cite[Theorem 1.1]{ru}
and \cite[Theorem 3.1]{jy}; we omit the details.

\begin{thm}\label{t3.1}
Let $\oz$ satisfy Assumption (C). Then for any $f\in
T_\oz(\cx)$, there exist $T_{\oz}(\cx)$-atoms $\{a_j\}_{j=1}^\fz$
and $\{\lz_j\}_{j=1}^\fz\subset \cc$ such that for almost
every $(x,\,t)\in \cx\times(0,\fz)$,
\begin{equation}\label{3.1}
f(x,t)=\sum_{j=1}^\fz\lz_ja_j(x,t),
\end{equation}
and the series converges in the space $T_{\oz}(\cx)$. Moreover,
there exists a positive constant $C$ such that for all $f\in
T_\oz(\cx)$,
\begin{equation}\label{3.2}
\Lambda(\{\lz_ja_j\}_j)\equiv\inf\lf\{\lz>0:
\ \sum_{j=1}^\fz V(B_j)\oz\lf(\frac{|\lz_j|}
{\lz V(B_j)\ro(V(B_j))}\r)\le1\r\}\le
C\|f\|_{T_\oz(\cx)},
\end{equation}
where $\widehat{B_j}$ appears as the support of $ a_j$.
\end{thm}

\begin{rem}\label{r3.1}\rm
(i) Let $\{\lz_j^i\}_{i,j}$ and $\{a_j^i\}_{i,j}$ satisfy
$\Lambda(\{\lz_j^ia_j^i\}_j)<\fz$, where $i=1,\,2.$ Since $\oz$ is
of strictly lower type $p_\oz$, we have
$[\Lambda(\{\lz_j^ia_j^i\}_{i,j})]^{p_\oz}\le
\sum_{i=1}^2[\Lambda(\{\lz_j^ia_j^i\}_j)]^{p_\oz}.$

(ii) Since $\oz$ is concave, it is of upper type 1. Thus,
$\sum_j|\lz_j|\ls \Lambda(\{\lz_ja_j\}_j) \ls \|f\|_{T_\oz(\cx)}.$
\end{rem}

The following conclusions on the convergence of $\eqref{3.1}$
play an important role in the remaining part of this paper.

\begin{cor}\label{c3.1} Let $\oz$ satisfy Assumption (C).
If $f\in T_2^2(\cx)\cap T_\oz(\cx)$, then the decomposition
\eqref{3.1} holds in both $T_\oz(\cx)$ and $T_2^2(\cx)$.
\end{cor}

The proof of Corollary \ref{c3.1} is similar to that of \cite[Proposition 3.1]{jy};
we omit the details.

In what follows, let $T_\oz^b(\cx)$ and $T^{p,b}_2(\cx)$ denote,
respectively, the spaces of all functions in $T_\oz(\cx)$ and
$T^{p}_2(\cx)$ with bounded support, where $p\in (0,\fz)$. Here and
in what follows, a function $f$ on $\cx\times (0,\fz)$ having
bounded support means that there exist a ball $B\subset\cx$ and
$0<c_1<c_2$ such that $\supp f\subset B\times (c_1, c_2)$.

\begin{lem}\label{l3.1}
(i) For all $p\in (0,\,\fz)$,
$T^{p,b}_2(\cx)\subset T_2^{2,b}(\cx).$
In particular, if $p\in (0,2]$, then
$T^{p,b}_2(\cx)$ coincides with $T_2^{2,b}(\cx)$.

(ii) Let $\oz$ satisfy Assumption (C). Then
$T^b_\oz(\cx)$ coincides with $T_2^{2,b}(\cx).$
\end{lem}

The proof of Lemma \ref{l3.1} is similar to that of \cite[Lemma 3.3]{jy}
and we omit the details.

\section{Orlicz-Hardy spaces and their dual spaces \label{s4}}

\hskip\parindent In this section, we always assume that the operator $L$ satisfies
Assumptions (A) and (B), and the Orlicz function $\oz$
satisfies Assumption (C). We introduce the Orlicz-Hardy space
associated to $L$ via the Lusin-area function and give its dual space
via the atomic and molecular decompositions of the Orlicz-Hardy space.
Let us begin with some notions and notation.

For all function $f\in L^2(\cx)$, the Lusin-area function $\cs_L(f)$
is defined by setting, for all $x\in\cx$,
\begin{equation}\label{4.1}
S_Lf(x)\equiv \bigg(\iint_{\Gamma(x)}|t^2Le^{-t^2L}f(y)|^2
\frac{\,d\mu(y)}{V(x,t)}\frac{\,dt}{t}\bigg)^{1/2}.
\end{equation}
From \eqref{2.8}, it follows that $\cs_L$ is bounded on $L^2(\cx)$.
Hofmann and Mayboroda \cite{hm1} introduced the Hardy space
$H_L^1(\rn)$ associated with a second order divergence form elliptic
operator $L$ as the completion of $\{f\in L^2(\rn):\ \cs_L(f)\in
L^1(\rn)\}$ with respect to the norm
$\|f\|_{H_L^1(\rn)}\equiv\|\cs_L(f)\|_{L^1(\rn)}.$ Similarly,
Hofmann et al \cite{hlmmy} introduced the Hardy space $H_L^1(\cx)$
associated to the nonnegative self-adjoint operator $L$ satisfying
the Davies-Gaffney estimate on metric measured spaces in the same
way.

Let $\car(L)$ denote the range of $L$ in $L^2(\cx)$ and $\nn(L)$ its
null space. Then $\overline{\car(L)}$ and $\nn(L)$ are orthogonal
and
\begin{equation}\label{4.2}
L^2(\cx)=\overline{\car(L)} \oplus \nn(L).
\end{equation}
Following \cite{amr,hlmmy}, we introduce the Orlicz-Hardy space
$H_{\oz,L}(\cx)$ associated to $L$ and $\oz$ as follows.

\begin{defn}\label{d4.1}
Let $L$ satisfy Assumptions (A) and (B) and $\oz$ satisfy Assumption
(C). A function $f\in \overline{\car(L)}$ is said to be in $\wz
H_{\oz,\,L}(\cx)$ if $\cs_L(f)\in L(\oz)$; moreover, define
$$\|f\|_{H_{\oz,L}(\cx)}\equiv \|\cs_L(f)\|_{L(\oz)}
\equiv\inf\lf\{\lz>0:\int_{\cx}\oz\lf
(\frac{\cs_L(f)(x)}{\lz}\r)\,d\mu(x)\le 1\r\}.$$
The Orlicz-Hardy space
$H_{\oz,L}(\cx)$  is defined to be the  completion of $\wz
H_{\oz,L}(\cx)$  in the norm $\|\cdot\|_{H_{\oz,L}(\cx)}.$
\end{defn}

\begin{rem}\label{r4.1}\rm
(i) Notice that for $0\neq f\in L^2(\cx)$, $\|\cs_L(f)\|_{L(\oz)}=0$
holds if and only if $f\in \nn(L)$. Indeed, if $f\in\nn(L)$, then
$t^2Le^{-t^2L}f=0$ and hence $\|\cs_L(f)\|_{L(\oz)}=0$. Conversely,
if $\|\cs_L(f)\|_{L(\oz)}=0$, then $t^2Le^{-t^2L}f=0$ for all $t\in
(0,\fz)$. Hence for all $t\in (0,\fz)$,
$(e^{-t^2L}-I)f=\int_0^t-2sLe^{-s^2L}f\,ds=0,$ which further implies that
$Lf=Le^{-t^2L}f=0$ and $f\in\nn(L)$. Thus, in Definition \ref{d4.1},
it is necessary to use $\overline{\car(L)}$ rather than $L^2(\cx)$
to guarantee $\|\cdot\|_{H_{\oz,L}(\cx)}$ to be a norm. For example,
if $\mu(\cx)<\fz$ and $e^{-tL}1=1$, then we have $1\in L^2(\cx)$ and
$L1=Le^{-tL}1=\frac{\,d}{\,dt}e^{-tL}1=0,$
which implies that $1\in\nn(L)$ and
$\|\cs_L(1)\|_{L(\oz)}=0$.

(ii) From the strictly lower type property of $\oz$, it is easy to
see that for all $f_1,\,f_2\in H_{\oz,L}(\cx)$,
$\|f_1+f_2\|_{H_{\oz,L}(\cx)}^{p_\oz}
\le \|f_1\|_{H_{\oz,L}(\cx)}^{p_\oz}
+\|f_2\|_{H_{\oz,L}(\cx)}^{p_\oz}.$

(iii) From the theorem of completion of Yosida \cite[p.\,56]{yo}, it follows that
$\wz H_{\oz,\,L}(\cx)$  is dense in $H_{\oz,\,L}(\cx)$,
namely, for any $f\in H_{\oz,\,L}(\cx)$, there exists a Cauchy sequence
$\{f_k\}^{\fz}_{k=1}$ in $\wz H_{\oz,\,L}(\cx)$
such that $\lim_{k\to\fz}\|f_k-f\|_{H_{\oz,\,L}(\cx)}=0.$

(iv) If $\oz(t)= t$ for all $t\in (0,\fz)$, then the space
$H_{\oz,\,L}(\cx)$ is just the
space $H^1_L(\cx)$ introduced by  Hofmann et al \cite{hlmmy}.
Moreover, if $\oz(t)=t^p$ for all $t\in (0,\fz)$, where $p\in (0,1]$,
we then denote the Orlicz-Hardy space $H_{\oz,L}(\cx)$ by $H_L^p(\cx)$.

(v) If $\cx=\rn$, $L=-\Delta$ and $\oz$ satisfies Assumption (C)
with $p_\oz \in (n/(n+1),1]$, then the Orlicz-Hardy space $H_{\oz,L}(\rn)$
coincides with the Orlicz-Hardy space in \cite{jyz}
and it was proved there that $H_{\oz,L}(\rn)=H_\oz(\rn)$;  see \cite{ja,v}
for the definition of $H_\oz(\rn)$.
\end{rem}

We now introduce the notions of $(\oz,M)$-atoms and $(\oz,M,\ez)$-molecules as follows.

\begin{defn}\label{d4.2}
Let $M\in\cn$. A function $\az\in L^2(\cx)$ is called an $(\oz,M)$-atom associated to
the operator $L$ if there exists a function $b\in \cd(L^M)$ and a ball
$B$ such that

(i) $\az=L^Mb$;

(ii) $\supp L^kb\subset B$, $k\in\{0,1,\cdots,M\}$;

(iii) $\|(r_B^2L)^kb\|_{L^2(\cx)}\le r_B^{2M}[V(B)]^{-1/2}[\ro(V(B))]^{-1},$
  $k\in \{0,1,\cdots,M\}$.
\end{defn}

\begin{defn}\label{d4.3}
Let $M\in\cn$ and $\ez\in(0,\fz)$.
A function $\bz\in L^2(\cx)$ is called an $(\oz,M,\ez)$-molecule associated to
the operator $L$ if there exist a function $b\in \cd(L^M)$ and a ball
$B$ such that

(i) $\bz=L^Mb$;

(ii) For every $k\in\{0,1,\cdots,M\}$ and $j\in {\zz}_+$, there holds
$$\|(r_B^2L)^kb\|_{L^2(U_j(B))}\le r_B^{2M}2^{-j\ez}
[V(2^jB)]^{-1/2}[\ro(V(2^jB))]^{-1},$$
where $U_j(B)$ for $j\in\zz_+$ is as in \eqref{2.5}.
\end{defn}

It is easy see that each $(\oz,M)$-atom is an $(\oz,M,\ez)$-molecule
for any $\ez\in (0,\fz)$.

\begin{prop}\label{p4.1}
Let $L$ satisfy Assumptions (A) and (B), $\oz$ satisfy Assumption
(C), $\ez>n(1/p_\oz-1/{p_\oz^+})$ and $M> \frac
n2(\frac{1}{p_\oz}-\frac 12)$. Then all $(\oz,M)$-atoms and
$(\oz,M,\ez)$-molecules are in $H_{\oz,L}(\cx)$ with norms bounded
by a positive constant.
\end{prop}

\begin{proof}
Since each $(\oz,M)$-atom is an $(\oz,M,\ez)$-molecule, we only need
to prove the proposition with an arbitrary  $(\oz,M,\ez)$-molecule
$\bz$ associated to a ball $B\equiv B(x_B,r_B)$.

Let $\wz p_\oz$ be as in Convention such that $\ez>n(1/p_\oz-1/{\wz
p_\oz})$ and $\lz\in\cc$. Then there exists $b\in L^2(\cx)$ such
that $\bz=L^M b$. Write
\begin{eqnarray*}&&\int_{\cx}\oz(\cs_L(\lz\bz)(x))\,d\mu(x)\\
&&\hs\le \int_{\cx}\oz(|\lz|\cs_L([I-e^{-r^2_BL}]^M\bz)(x))\,d\mu(x)
+\int_{\cx}\oz(|\lz|\cs_L((I-[I-e^{-r^2_BL}]^M)\bz)(x))\,d\mu(x)\\
&&\hs\le\sum_{j=0}^\fz \int_{\cx}\oz(|\lz|\cs_L([I-e^{-r^2_BL}]^M
(\bz\chi_{U_j(B)}))(x))\,d\mu(x)\\
&&\hs\hs+\sum_{j=0}^\fz
\int_{\cx}\oz(|\lz|\cs_L((I-[I-e^{-r^2_BL}]^M)
(L^M[b\chi_{U_j(B)}]))(x))\,d\mu(x) \equiv\sum_{j=0}^\fz
\mathrm{H}_j+\sum_{j=0}^\fz\mathrm{I}_j.
\end{eqnarray*}
Let us estimate the first term. For each $j\ge 0$, let $B_j\equiv
2^jB$ in this proof. Since $\oz$ is concave, by the Jensen
inequality and the H\"older inequality, we obtain
\begin{eqnarray*}
\mathrm{H}_j&&\le \sum_{k=0}^\fz \int_{U_k(B_j)}\oz(|\lz|
\cs_L([I-e^{-r^2_BL}]^M(\bz\chi_{U_j(B)}))(x))\,d\mu(x)\\
&&\le \sum_{k=0}^\fz
V(2^kB_j)\oz\bigg(\frac{|\lz|}{V(2^kB_j)}\int_{U_k(B_j)}
\cs_L([I-e^{-r^2_BL}]^M(\bz\chi_{U_j(B)}))(x)\,d\mu(x)\bigg)\\
&&\le \sum_{k=0}^\fz
V(2^kB_j)\oz\bigg(\frac{|\lz|}{[V(2^kB_j)]^{1/2}}
\|\cs_L([I-e^{-r^2_BL}]^M(\bz\chi_{U_j(B)}))\|_{L^2(U_k(B_j))}\bigg).
\end{eqnarray*}
For $k=0,1,2$, by the $L^2(\cx)$-boundedness of $\cs_L$ and
$e^{-r^2_BL}$,  we obtain
\begin{equation}\label{4.3}
\|\cs_L([I-e^{-r^2_BL}]^M(\bz\chi_{U_j(B)}))\|_{L^2(U_k(B_j))}\ls
\|\bz\|_{L^2(U_j(B))}.
\end{equation}
The proof of the case $k\ge 3$
involves much more complicated calculation, which is similar to the
proof of \cite[Lemma 4.2]{hm1}. We give the details for the
completeness. Write
\begin{eqnarray*}
&&\|\cs_L([I-e^{-r^2_BL}]^M
(\bz\chi_{U_j(B)}))\|^2_{L^2(U_k(B_j))}\\
&&\hs\ls \iint_{\car(U_k(B_j))}|t^2Le^{-t^2L}[I-e^{-r^2_BL}]^M
(\bz\chi_{U_j(B)})(x)|^2\frac{\,d\mu(x)\,dt}{t}\\
&&\hs\ls \int_0^\fz\int_{\rn\setminus
2^{k-2}B_j}\lf|t^2Le^{-t^2L}[I-e^{-r^2_BL}]^M
(\bz\chi_{U_j(B)})(x)\r|^2\frac{\,d\mu(x)\,dt}{t}\\
&&\hs\hs+\sum_{i=0}^{k-2}\int_{(2^{k-1}-2^i)2^jr_B}^\fz\int_{U_i(B_j)}
\cdots\frac{\,d\mu(x)\,dt}{t}\equiv \mathrm{J}+\sum_{i=0}^{k-2}\mathrm{J}_i.
\end{eqnarray*}
Using the fact that $I-e^{-r_B^2L}=\int_{0}^{r_B^2}Le^{-sL}\,ds$,
Lemma \ref{l2.1} and the Minkowski inequality, we obtain
\begin{eqnarray*}
\mathrm{J}&&=\int_0^\fz\int_{\rn\setminus
2^{k-2}B_j}\bigg|\int_0^{r_B^2}\cdots\int_0^{r_B^2}t^2L^{M+1}
e^{-(t^2+s_1+\cdots+s_M)L}\\
&&\hs \times(\bz\chi_{U_j(B)})(x)\,ds_1\cdots\,ds_M\bigg|^2\frac{\,d\mu(x)\,dt}{t}\\
&&\ls \lf\{\int_0^{r_B^2}\cdots\int_0^{r_B^2}\lf[\int_0^\fz
\frac{t^4\|\bz\|_{L^2(U_j(B))}^2}{(t^2+s_1+\cdots+s_M)^{2(M+1)}}\r.\r.\\
&&\hs\lf.\lf.\times \exp\lf\{-\frac{\dist(B_j, \rn\setminus
2^{k-1}B_j)^2}{t^2+s_1+\cdots+s_M}\r\}\frac{\,dt}{t}\r]^{1/2}
\,ds_1\cdots\,ds_M\r\}^2\\
&&\ls r_B^{4M}\|\bz\|^2_{L^2(U_j(B))}\int_0^\fz
(2^{k+j}r_B)^{-4M}\min\lf\{\frac{2^{k+j}r_B}{t },\,\frac{t}{
2^{k+j}r_B}\r\}\frac{\,dt}{t}\\
&&\ls 2^{-4M(k+j)}\|\bz\|^2_{L^2(U_j(B))}.
\end{eqnarray*}
Similarly,
\begin{eqnarray*}
\sum_{i=0}^{k-2}\mathrm{J}_i&&=\sum_{i=0}^{k-2}
\int_{U_i(B_j)}\int_{(2^{k-1}-2^i)2^jr_B}^\fz
\bigg|\int_0^{r_B^2}\cdots\int_0^{r_B^2}t^2L^{M+1}
e^{-(t^2+s_1+\cdots+s_M)L}\\
&&\hs \times(\bz\chi_{U_j(B)})(x)\,ds_1\cdots\,ds_M\bigg|^2\frac{\,d\mu(x)\,dt}{t}\\
&&\ls\sum_{i=0}^{k-2}\lf\{\int_0^{r_B^2}\cdots\int_0^{r_B^2}
\lf[\int_{2^{k+j-2}r_B}^\fz
\frac{t^4\|\bz\|_{L^2(U_j(B))}^2}{(t^2+s_1+\cdots+s_M)^{2(M+1)}}
\frac{\,dt}{t}\r]^{1/2}\,ds_1\cdots\,ds_M\r\}^2\\
&&\ls (k-2)2^{-4M(k+j)}\|\bz\|^2_{L^2(U_j(B))}.
\end{eqnarray*}
Combining the estimates of $\mathrm{J}$ and $\mathrm{J}_i$, we
obtain that
\begin{eqnarray}\label{4.4}\|\cs_L([I-e^{-r^2_BL}]^M
(\bz\chi_{U_j(B)}))\|_{L^2(U_k(B_j))}\ls
\sqrt k2^{-2M(k+j)}\|\bz\|_{L^2(U_j(B))}.
\end{eqnarray}
By Definition \ref{d4.3}, $2Mp_\oz>n(1-p_\oz/2)$, Assumption (C),
\eqref{4.3} and \eqref{4.4}, we have
\begin{eqnarray*}
\mathrm{H}_j&&\ls
V(B_j)\oz\bigg(\frac{|\lz|2^{-j\ez}}{V(B_j)\ro(V(B_j)) }\bigg)+
\sum_{k=3}^\fz V(2^kB_j)\oz\bigg(\frac{|\lz|\sqrt k
2^{-{2M(j+k)}-j\ez}}
{[V(2^kB_j)]^{1/2} [V(B_j)]^{1/2}\ro(V(B_j)) }\bigg)\\
&&\ls 2^{-jp_\oz\ez}V(B_j)\oz\bigg(\frac{|\lz|}{V(B_j)\ro(V(B_j)) }\bigg)\\
&&\hs\hs+\sum_{k=3}^\fz \sqrt
k2^{kn(1-p_\oz/2)}2^{-{2Mp_\oz(j+k)}-jp_\oz\ez}
V(B_j)\oz\bigg(\frac{|\lz|}{V(B_j)\ro(V(B_j)) }\bigg)\\
&&\ls 2^{-jp_\oz\ez} V(B_j)\oz\bigg(\frac{|\lz|}{V(B_j)\ro(V(B_j))}\bigg).
\end{eqnarray*}
Since $\ro$ is of lower type $1/{\wz p_\oz}-1$ and
$\ez>n(1/p_\oz-1/{\wz p_\oz})$, we further obtain
\begin{eqnarray*}
\sum_{j=0}^\fz \mathrm{H}_j&&\ls\sum_{j=0}^\fz 2^{-jp_\oz\ez}
V(B_j)\lf\{\frac{V(B)\ro(V(B))}{V(B_j)\ro(V(B_j))}\r\}^{p_\oz}
\oz\bigg(\frac{|\lz|}{V(B)\ro(V(B)) }\bigg)\\
&& \ls\sum_{j=0}^\fz
2^{-jp_\oz\ez}V(B_j)\lf\{\frac{V(B)}{V(B_j)}\r\}^{p_\oz/{\wz p_\oz}}
\oz\bigg(\frac{|\lz|}{V(B)\ro(V(B)) }\bigg)\\
&&\ls\sum_{j=0}^\fz 2^{-jp_\oz\ez} 2^{jn(1-p_\oz/{\wz p_\oz})}V(B)
\oz\bigg(\frac{|\lz|}{V(B)\ro(V(B)) }\bigg) \ls
V(B)\oz\bigg(\frac{|\lz|}{V(B)\ro(V(B)) }\bigg).
\end{eqnarray*}

Let us now estimate the remaining term $\{\mathrm{I}_j\}_{j\ge 0}.$
Applying the Jensen inequality, we have
\begin{eqnarray*}
\mathrm{I}_j&&\ls \sum_{k=0}^\fz
\int_{U_k(B_j)}\oz(|\lz|\cs_L((I-[I-e^{-r^2_BL}]^M)
(L^M[b\chi_{U_j(B)}]))(x))\,d\mu(x)\\
&&\ls \sum_{k=0}^\fz
V(2^kB_j)\oz\bigg(\frac{|\lz|}{[V(2^kB_j)]^{1/2}}
\|\cs_L((I-[I-e^{-r^2_BL}]^M)
(L^M[b\chi_{U_j(B)}]))\|_{L^2(U_k(B_j))}\bigg).
\end{eqnarray*}
Notice that
\begin{eqnarray*}
&&\|\cs_L((I-[I-e^{-r^2_BL}]^M)
(L^M[b\chi_{U_j(B)}]))\|_{L^2(U_k(B_j))}\\
&&\hs\ls r_B^{-2M}\sup_{1\le l\le M}\|\cs_L((lr_B^2L)^Me^{-l r_B^2L}
[b\chi_{U_j(B)}]))\|_{L^2(U_k(B_j))}.
\end{eqnarray*}
For $k=0,1,2$, by the $L^2(\cx)$-boundedness of $\cs_L$ and
$(lr_B^2L)^Me^{-l r_B^2L},$ we have
\begin{eqnarray*}
\|\cs_L((lr_B^2L)^Me^{-l r_B^2L}
[b\chi_{U_j(B)}]))\|_{L^2(U_k(B_j))}\ls \|b\|_{L^2(U_j(B))}.
\end{eqnarray*}
For $k\ge 3$, Lemma \ref{l2.1} yields that
\begin{eqnarray*}
&&\|\cs_L((lr_B^2L)^Me^{-l r_B^2L}
[b\chi_{U_j(B)}]))\|^2_{L^2(U_k(B_j))}\\
&&\hs\ls
r^{4M}_B\iint_{\car(U_k(B_j))}|t^2L^{M+1}e^{-(t^2+lr_B^2)L}
[b\chi_{U_j(B)}](x)|^2\frac{\,d\mu(x)\,dt}{t}\\
&&\hs\ls r^{4M}_B\bigg\{\int_0^\fz\int_{\rn\setminus
2^{k-2}B_j}\lf|\frac{t^2[(t^2+lr_B^2)L]^{M+1}e^{-(t^2+lr_B^2)L}
[b\chi_{U_j(B)}](x)}{(t^2+lr_B^2)^{M+1}}\r|^2\frac{\,d\mu(x)\,dt}{t}\\
&&\hs\hs+\int_{(2^{k-1}-2^i)2^jr_B}^\fz\sum_{i=0}^{k-2}\int_{U_i(B_j)}
\cdots\frac{\,d\mu(x)\,dt}{t}\bigg\}\\
&&\hs\ls r^{4M}_B \|b\|_{L^2(U_j(B))}^2\bigg[\int_0^\fz
\frac{t^4}{(t^2+lr_B^2)^{2(M+1)}}\exp\lf\{-\frac{\dist(B_j,
\rn\setminus 2^{k-1}B_j)^2}{t^2+l r_B^2}\r\}\frac{\,dt}{t}\\
&&\hs\hs+(k-2)\int_{2^{k-2+j}r_B}^\fz
\frac{\,dt}{t^{4M+1}}\bigg]\
\ls k2^{-4M(k+j)}\|b\|_{L^2(U_j(B))}^2.
\end{eqnarray*}

Combining the above estimates, similarly to the calculation of
$\mathrm{H}_j$, we obtain
\begin{eqnarray*}
\sum_{j=0}^\fz \mathrm{I}_j&&\ls
V(B)\oz\bigg(\frac{|\lz|}{V(B)\ro(V(B)) }\bigg),
\end{eqnarray*}
which further yields that
\begin{eqnarray}\label{4.5}
\int_{\cx}\oz(\cs_L(\lz\bz)(x))\,d\mu(x)\ls
V(B)\oz\bigg(\frac{|\lz|}{V(B)\ro(V(B)) }\bigg).
\end{eqnarray}
This implies that $\|\bz\|_{H_{\oz,L}(\cx)}\ls 1$, which completes
the proof of Proposition \ref{p4.1}.
\end{proof}

\subsection{Decompositions into atoms and molecules\label{s4.1}}

\hskip\parindent In what follows, let $M\in\cn$ and
$M>\frac n2(\frac1{p_\oz}-\frac 12)$, where $p_\oz$ is as in Assumption (C).
We also let $\Phi$ be as in Lemma \ref{l2.2} and
 $\Psi(t)=t^{2(M+1)}\Phi(t)$ for all $t\in (0,\fz)$.
 For all $f\in L^2_b(\cx\times(0,\fz))$ and $x\in\cx$, define
\begin{equation}\label{4.6}
\pi_{\Psi,L}f(x)\equiv C_\Psi\int_0^\fz
\Psi(t\sqrt L)(f(\cdot,t))(x)\frac{\,dt}{t},
\end{equation}
where $C_{\Psi}$ is the positive constant such that
\begin{equation}\label{4.7}
C_\Psi\int_0^\fz \Psi(t)t^2e^{-t^2}\frac{\,dt}{t}=1.
\end{equation}
Here $L^2_b(\cx\times(0,\fz))$ denotes the space of all function
$f\in L^2(\cx\times(0,\fz))$ with bounded support. Recall that a
function $f$ on $\cx\times (0,\fz)$ having bounded support means
that there exist a ball $B\subset\cx$ and $0<c_1<c_2$ such that
$\supp f\subset B\times (c_1, c_2)$.

\begin{prop}\label{p4.2} Let $L$ satisfy Assumptions (A) and (B),
$\oz$ satisfy Assumption (C), $M>\frac n2(\frac1{p_\oz}-\frac 12)$
 and $\pi_{\Psi,L}$ be as in \eqref{4.6}. Then

(i) the operator $\pi_{\Psi,L}$, initially defined on the space
$T_2^{2,b}(\cx)$, extends to a  bounded linear operator from
$T_2^2(\cx)$ to $L^2(\cx)$;

(ii) the operator $\pi_{\Psi,L}$, initially defined on the space
$T_\oz^{b}(\cx)$, extends to a bounded linear operator from $T_\oz(\cx)$ to
$H_{\oz,L}(\cx).$
\end{prop}

\begin{proof}
(i) Suppose that $f\in T^{2,b}_2(\cx)$. For any $g\in L^2(\cx)$, by
the H\"older inequality and \eqref{2.8}, we have
\begin{eqnarray*}
|\la \pi_{\Psi,L}(f),g\ra|&&=\bigg|C_{\Psi}\int_0^\fz \la
\Psi(t\sqrt L)f, g\ra \frac{\,dt}{t}\bigg|\\
&&\ls \int_{\cx}\int_{\Gamma(x)} \bigg|f(y,t)
\Psi(t\sqrt L)g(y,t)\bigg|\frac{\,d\mu(y)}{V(y,t)}\frac{\,dt}{t}\,d\mu(x)\\
&&\ls \int_{\cx}\ca(f)(x)\ca(\Psi(t\sqrt L)g)(x)\,d\mu(x)
\ls \|f\|_{T^2_2(\cx)}\|g\|_{L^2(\cx)},
\end{eqnarray*}
which implies that $\|\pi_{\Psi,L}(f)\|_{L^2(\cx)}\ls \|f\|_{T^2_2(\cx)}$.
From this and the density of $T^{2,b}_2(\cx)$ in $T^2_2(\cx)$,
we deduce (i).

To prove (ii), let $f\in T_\oz^b(\cx)$. Then, by Lemma \ref{l3.1}(ii),
Corollary \ref{c3.1} and (i) of this proposition, we have
$$\pi_{\Psi,L}(f)=\sum_{j=1}^\fz\lz_j\pi_{\Psi,L}(a_j)\equiv\sum_{j=1}^\fz\lz_j\az_j$$
in $L^2(\cx)$, where $\{\lz_j\}_{j=1}^\fz$ and $\{a_j\}_{j=1}^\fz$ satisfy \eqref{3.2}.
Recall that here, $\supp a_j\subset \widehat B_j$ and $B_j$ is a ball of $\cx$.

On the other hand, by \eqref{2.8}, we have that the operator $\cs_L$ is
bounded on $L^2(\cx)$, which implies that for all $x\in \cx$,
$\cs_L(\pi_{\Psi,L}(f))(x)\le \sum_{j=1}^\fz|\lz_j|\cs_L(\az_j)(x).$
This, combined with the monotonicity, continuity and subadditivity of $\oz$,
yields that
$$\int_\cx \oz(\cs_L(\pi_{\Psi,L}(f))(x))\,d\mu(x)\le
\sum_{j=1}^\fz\int_\cx \oz(|\lz_j|\cs_L(\az_j)(x))\,d\mu(x).$$

We now show that $\az_j=\pi_{\Psi,L}(a_j)$ is a multiple of an
$(\oz,M)$-atom for each $j$. Let
$$b_j\equiv C_\Psi \int_0^\fz t^{2M}t^2L\Phi(t\sqrt L)(a_j(\cdot,t))\,\frac{\,dt}{t}.$$
Then $\az_j=L^M b_j$. Moreover, by Lemma \ref{l2.2}, for each $k\in
\{0,1,\cdots,M\}$, we have $\supp L^kb_j\subset B_j$. On the other
hand, for any $h\in L^2(B_j)$, by the H\"older inequality and
\eqref{2.8}, we have
\begin{eqnarray*}
&&\bigg|\int_\cx (r_{B_j}^2L)^kb_j(x)h(x)\,d\mu(x)\bigg|\\
&&\hs =C_\Psi\bigg|\int_\cx\int_0^\fz t^{2M}
(r_{B_j}^2L)^kt^2L\Phi(t\sqrt L)(a_j(\cdot,t))(x)h(x)\frac{\,d\mu(x)\,dt}{t}\bigg|\\
&&\hs\ls r_{B_j}^{2M}\int_\cx\int_0^\fz \bigg|a_j(y,t)
(t^2L)^{k+1}\Phi(t\sqrt L)h(y)\bigg|\frac{\,d\mu(y)\,dt}{t}\\
&&\hs\ls r_{B_j}^{2M}\|a_j\|_{T^2_2(\cx)}\lf(\int_\cx\int_0^\fz
|(t^2L)^{k+1}\Phi(t\sqrt L)h(y)|^2\frac{\,d\mu(y)\,dt}{t}\r)^{1/2}\\
&&\hs \ls r_{B_j}^{2M} [V(B_j)]^{-1/2}[\ro(V(B_j))]^{-1}\|h\|_{L^2(\cx)},
\end{eqnarray*}
which implies that $\az_j$ is an $(\oz,M)$-atom up to a harmless
constant.

By \eqref{4.5}, we obtain
\begin{eqnarray*}\int_{\cx}\oz(\cs_L(\pi_{\Psi,L}(f))(x))\,d\mu(x)
&&\le\sum_{j=1}^\fz\int_\cx
\oz(|\lz_j|\cs_L(\az_j)(x))\,d\mu(x)\\
&&\ls \sum_{j=1}^\fz V(B_j)\oz\lf(\frac{|\lz_j|}
{V(B_j)\ro(V(B_j))}\r),
\end{eqnarray*}
which implies that $\|\pi_{\Psi,L}(f)\|_{H_{\oz,L}(\cx)}
\ls \Lambda(\{\lz_ja_j\}_j)\ls
\|f\|_{T_\oz(\cx)},$ and hence completes the proof of Proposition
\ref{p4.2}.
\end{proof}

\begin{prop}\label{p4.3}
Let $L$ satisfy Assumptions (A) and (B), $\oz$ satisfy Assumption
(C) and $M> \frac n2(\frac{1}{p_\oz}-\frac 12)$. Then for all $f\in
H_{\oz,L}(\cx)\cap L^2(\cx)$, there exist $(\oz,M)$-atoms
$\{\az_j\}_{j=1}^\fz$ and $\{\lz_j\}_{j=1}^\fz\subset \cc$ such that
$  f=\sum_{j=1}^\fz\lz_j\az_j $ in both $H_{\oz,L}(\cx)$ and
$L^2(\cx)$. Moreover, there exists a positive constant $C$ such that
for all $f\in H_{\oz,L}(\cx)\cap L^2(\cx)$,
\begin{equation*}
\Lambda(\{\lz_j\az_j\}_j)\equiv\inf\lf\{\lz>0:\,\sum_{j=1}^\fz
V(B_j)\oz\lf(\frac{|\lz_j|}{\lz
V(B_j)\ro(V(B_j))}\r)\le1\r\}\le C\|f\|_{H_{\oz,L}(\cx)},
\end{equation*}
where for each $j$, $\az_j$ is supported in the ball $B_j$.
\end{prop}

\begin{proof}
Let $f\in H_{\oz,L}(\cx)\cap L^2(\cx)$. Then by $H_\fz$-functional
calculus for $L$ together with \eqref{4.7}, we have
\begin{equation*}
f=C_{\Psi}\int^\fz_0 \Psi(t\sqrt L)t^2Le^{-t^2L}f\frac{\,dt}{t}=
\pi_{\Psi,L}(t^2Le^{-t^2L}f)
\end{equation*}
in $L^2(\cx).$ By Definition \ref{d4.1} and \eqref{2.8}, we have
$t^2Le^{-t^2L}f\in T_\oz(\cx)\cap T_2^2(\cx)$. Applying Theorem
\ref{t3.1}, Corollary \ref{c3.1} and Proposition \ref{p4.2} to
$t^2Le^{-t^2L}f$, we obtain
\begin{equation*}
f=\pi_{\Psi,L}(t^2Le^{-t^2L})=\sum_{j=1}^\fz\lz_j\pi_{\Psi,L}(a_j)
\equiv\sum_{j=1}^\fz\lz_j\az_j
\end{equation*}
in both $L^2(\cx)$ and
$H_{\oz,L}(\cx)$, and
$\Lambda(\{\lz_ja_j\}_j)
\ls\|t^2Le^{-t^2L}f\|_{T_\oz(\cx)}\sim\|f\|_{H_{\oz,L}(\cx)}.$

On the other hand, by the proof of Proposition \ref{p4.2}, we obtain
that for each $j\in\cn$, $\az_j$ is an $(\oz,M)$-atom up to a
harmless constant, which completes the proof of Proposition
\ref{p4.3}.
\end{proof}

From Proposition \ref{p4.3}, similarly to the proof of
\cite[Corollary 4.1]{jy}, we deduce the following result. We omit
the details.
\begin{cor}\label{c4.1}
Let $L$ satisfy Assumptions (A) and (B), $\oz$ satisfy Assumption
(C) and $M> \frac n2(\frac{1}{p_\oz}-\frac 12)$. Then for all $f\in
H_{\oz,L}(\cx)$, there exist $(\oz,M)$-atoms $\{\az_j\}_{j=1}^\fz$
and $\{\lz_j\}_{j=1}^\fz\subset \cc$ such that
$f=\sum_{j=1}^\fz\lz_j\az_j$ in $H_{\oz,L}(\cx)$.
Moreover, there exists a positive constant $C$ independent of $f$ such that
$\Lambda(\{\lz_j\az_j\}_j)\le  C\|f\|_{H_{\oz,L}(\cx)}.$
\end{cor}

Let $H_{\oz,\,{\rm fin}}^{{\rm at},\,M}(\cx)$ and $H_{\oz,\,{\rm
fin}}^{{\rm mol},\,M,\,\ez}(\cx)$ denote the spaces of finite
combinations of $(\oz,M)$-atoms and $(\oz,M,\ez)$-molecules,
respectively. From Corollary \ref{c4.1} and Proposition \ref{p4.1},
we deduce the following density conclusions.

\begin{cor}\label{c4.2}
Let $L$ satisfy Assumptions (A) and (B), $\oz$ satisfy Assumption
(C), $\ez>n(1/p_\oz-1/{p_\oz^+})$ and $M> \frac n2(\frac{1}{p_\oz}-\frac 12)$.
Then both the spaces $H_{\oz,\,{\rm fin}}^{{\rm at},\,M}(\cx)$
and $H_{\oz,\,{\rm fin}}^{{\rm mol},\,M,\,\ez}(\cx)$
are dense in the space $H_{\oz,L}(\cx)$.
\end{cor}

\subsection{Dual spaces of  Orlicz-Hardy spaces\label{s4.2}}

\hskip\parindent In this subsection, we study the dual space of the Orlicz-Hardy
 space $H_{\oz,L}(\cx)$. We begin with some notions.

Let $\phi=L^M\nu$ be a function in $L^2(\cx)$, where $\nu\in\cd(L^M)$.
Following \cite{hm1,hlmmy}, for $\ez>0$, $M\in\cn$ and
fixed $x_0\in\cx$, we introduce the space
$$\cm_\oz^{M,\ez}(L)\equiv \{\phi=L^M\nu\in L^2(\cx):
\ \|\phi\|_{\cm_\oz^{M,\ez}(L)}<\fz\},$$
where
$$\|\phi\|_{\cm_\oz^{M,\ez}(L)}\equiv \sup_{j\in\zz_+}
\lf\{2^{j\ez}[V(x_0,2^j)]^{1/2}
\ro(V(x_0,2^j))\sum_{k=0}^M\|L^k\nu\|_{L^2(U_j(B(x_0,1)))}\r\}.$$

Notice that if $\phi\in \cm_\oz^{M,\ez}(L)$ for some $\ez>0$ with
norm 1, then $\phi$ is an
$(\oz,M,\ez)$-molecule adapted to the ball $B(x_0,1)$. Conversely,
if $\bz$ is an $(\oz,M,\ez)$-molecule adapted to any ball, then
$\bz\in \cm_\oz^{M,\ez}(L)$.

Let $A_t$ denote either $(I+t^2L)^{-1}$ or $e^{-t^2L}$ and $f\in
(\cm_\oz^{M,\ez}(L))^\ast$, the dual space of $\cm_\oz^{M,\ez}(L)$.
We claim that $(I-A_t)^Mf\in L^2_{\loc}(\cx)$ in the sense of
distributions. In fact, for any ball $B$, if $\psi\in L^2(B)$, then
it follows from the Davies-Gaffney estimate \eqref{2.6} that
$(I-A_t)^M\psi\in \cm_\oz^{M,\ez}(L)$ for every $\ez>0$. Thus,
\begin{eqnarray*}
|\la (I-A_t)^Mf,\psi\ra|\equiv|\la f,(I-A_t)^M\psi\ra|\le
C(t,r_B,\dist(B,x_0))\|f\|_{(\cm_\oz^{M,\ez}(L))^\ast}\|\psi\|_{L^2(B)},
\end{eqnarray*}
which implies that $(I-A_t)^Mf\in L^2_{\loc}(\cx)$ in the sense of distributions.

Finally, for any $M\in \cn$, define
$$\cm_\oz^M(\cx)\equiv \bigcap_{\ez>n(1/p_\oz-1/p_\oz^{+})}
(\cm_\oz^{M,\ez}(L))^\ast.$$

\begin{defn}\label{d4.4}
Let $L$ satisfy Assumptions (A) and (B), $\oz$ satisfy Assumption (C),
 $\ro$ be as in \eqref{2.10} and  $M>\frac n2(\frac{1}{p_\oz}-\frac 12)$.
 A functional  $f\in\cm_\oz^M(\cx)$ is said to be in $\bmoo$ if
\begin{equation*}
\|f\|_{\bmoo}\equiv\sup_{B\subset\cx}\frac{1}{\ro(V(B))}\lf[\frac{1}{V(B)}\int_B
|(I-e^{-r_B^2L})^Mf(x)|^2 \,d\mu(x)\r]^{1/2}< \fz,
\end{equation*}
where the supremum is taken over all ball $B$ of $\cx.$
\end{defn}

The proofs of the following two propositions are similar to those
of Lemmas 8.1 and 8.3 of \cite{hm1}, respectively; we omit
the details.

\begin{prop}\label{p4.4}
Let $L$, $\oz$, $\ro$ and $M$ be as in Definition \ref{d4.4}.
Then $f\in \bmoo$ if and only if $f\in\cm_\oz^M(\cx)$ and
\begin{equation*}
\sup_{B\subset\cx}\frac{1}{\ro(V(B))}\lf[\frac{1}{V(B)}\int_B
|(I-(I+r_B^2L)^{-1})^Mf(x)|^2 \,d\mu(x)\r]^{1/2}<\fz.
\end{equation*}
Moreover, the quantity appeared in the left-hand side of the above formula
is equivalent to $\|f\|_{\bmoo}$.
\end{prop}

\begin{prop}\label{p4.5}
Let $L$, $\oz$, $\ro$ and $M$ be as in Definition \ref{d4.4}. Then
there exists a positive constant $C$ such that for all $f\in \bmoo$,
$$\sup_{B\subset \cx}\frac{1}{\ro(V(B))}\lf[\frac{1}{V(B)}\int_{\widehat B}
|(t^2L)^Me^{-t^2L}f(x)|^2 \frac{\,d\mu(x)\,dt}{t}\r]^{1/2}\le C\|f\|_{\bmoo}.$$
\end{prop}

The following Proposition \ref{p4.6} and Corollary \ref{c4.3} are a
kind of Calder\'on reproducing formulae.

\begin{prop}\label{p4.6}Let $L$, $\oz$, $\ro$ and $M$
be as in Definition \ref{d4.4},
$\ez>0$ and $\wz M>M+\ez+\frac n4+ \frac N2 (\frac 1{p_\oz}-1)$,
where $N$ is as in \eqref{2.4}.
Fix $x_0\in\cx$. Assume that $f\in \cm_\oz^M(\cx)$ satisfies
\begin{equation}\label{4.8}
\int_{\cx}\frac{|(I-(I+L)^{-1})^Mf(x)|^2}
{1+[d(x,x_0)]^{n+\ez+2N(1/p_\oz-1)}}\,d\mu(x)<\fz.
\end{equation}
Then for all $(\oz,\wz M)$-atom $\az$,
\begin{equation*}
\la f, \az\ra =\wz C_M\int_{\cx\times(0,\fz)} (t^2L)^Me^{-t^2L}f(x)
\overline{t^2Le^{-t^2L}\az(x)}\frac{\,d\mu(x)\,dt}{t},
\end{equation*}
where $\wz C_M$ is the positive constant satisfying $\wz
C_M\int_0^\fz t^{2(M+1)}e^{-2t^2}\frac{\,dt}{t}=1.$
\end{prop}

\begin{proof}
Let $\az$ be an $(\oz,\wz M)$-atom supported in $B\equiv
B(x_B,r_B)$. Notice that \eqref{4.8} implies that
\begin{equation*}
\int_{\cx}\frac{|(I-(I+L)^{-1})^Mf(x)|^2}
{r_B+[d(x,x_B)]^{n+\ez+2N(1/p_\oz-1)}}\,d\mu(x)<\fz.
\end{equation*}
For $R>\dz>0$, write
\begin{eqnarray*}
&&\wz C_M\int_\dz^R\int_\cx (t^2L)^Me^{-t^2L}f(x)
\overline{t^2Le^{-t^2L}\az(x)}\frac{\,d\mu(x)\,dt}{t}\\
&&\hs =\lf\la f, \wz C_M\int_\dz^R
(t^2L)^{M+1}e^{-2t^2L}\az\frac{\,dt}{t}\r\ra=\la f,\az\ra-\lf\la f,
\az-\wz C_M\int_\dz^R
(t^2L)^{M+1}e^{-2t^2L}\az\frac{\,dt}{t}\r\ra.\nonumber
\end{eqnarray*}
Since $\az$ is an $(\oz, \wz M)$-atom, by Definition
\ref{d4.2}, there exists $b\in L^2(\cx)$ such that $\az=L^{\wz M}b$.
Thus, by the fact that $\wz M>M+\frac n4+\frac N2 (\frac 1{p_\oz}-1)+\ez$, we obtain
\begin{eqnarray*}
\az&&=L^{\wz M}b=(L-L(I+L)^{-1}+L(I+L)^{-1})^M L^{\wz M-M}b\\
&&=\sum_{k=0}^M C_M^k
(L-L(I+L)^{-1})^{M-k}(L(I+L)^{-1})^k  L^{\wz M-M}b\\
&&=\sum_{k=0}^MC_M^k(I-(I+L)^{-1})^M L^{\wz M-k}b,
\end{eqnarray*}
where $C_M^k$ denotes the combinatorial number, which together
with $H_\fz$-functional calculus further implies that
\begin{eqnarray*}
&&\lf\la f, \az-\wz C_M\int_\dz^R
(t^2L)^{M+1}e^{-2t^2L}\az\frac{\,dt}{t}\r\ra\\
&&\hs =\sum_{k=0}^M C_M^k \lf\la (I-(I+L)^{-1})^M f, L^{\wz M-k}b-
\wz C_M\int_\dz^R (t^2L)^{M+1}e^{-2t^2L}L^{\wz M-k}b\frac{\,dt}{t}\r\ra\\
&&\hs =\sum_{k=0}^MC_M^k \lf\la (I-(I+L)^{-1})^M f,
\wz C_M\int_0^\dz (t^2L)^{M+1}e^{-2t^2L}L^{\wz M-k}b\frac{\,dt}{t}\r\ra\\
&& \hs\hs+\sum_{k=0}^MC_M^k\lf\la (I-(I+L)^{-1})^M f,
\wz C_M\int_R^\fz (t^2L)^{M+1}e^{-2t^2L}L^{\wz M-k}b\frac{\,dt}{t}\r\ra
\equiv \mathrm{H}+\mathrm{I}.
\end{eqnarray*}
By \eqref{4.8}, we see that up to a harmless constant, the term
$\mathrm{I}$ is bounded by
\begin{eqnarray*}
&&\lf\{\int_{\cx}\frac{|(I-(I+L)^{-1})^Mf(x)|^2}
{r_B+[d(x,x_B)]^{n+\ez+2N(\frac 1{p_\oz}-1)}}\,d\mu(x)\r\}^{1/2}
\sup_{0\le k\le M}\lf\{\int_\cx\bigg|\int_R^\fz (t^2L)^{M+1}\r.\\
&&\hs\hs\lf.\times e^{-2t^2L}
L^{\wz M-k}b(x)\frac{\,dt}{t}\bigg|^2(r_B+[d(x,x_B)]^{n+\ez+2N(\frac 1{p_\oz}-1)})
\,d\mu(x)\r\}^{1/2}\\
&&\hs\ls \sup_{0\le k\le M}\sum_{j=0}^\fz (2^jr_B)^{\frac {n+\ez}2
+N(\frac 1{p_\oz}-1)}
\int_R^\fz \|(t^2L)^{\wz M+M+1-k}e^{-2t^2L}
b\|_{L^2(U_j(B))}\frac{\,dt}{t^{2(\wz M-k)+1}}\\
&&\hs\ls  \sup_{0\le k\le M}\sum_{j=0}^2 (2^jr_B)^{\frac {n+\ez}2
+N(\frac 1{p_\oz}-1)}
\|b\|_{L^2(\cx)}\int_R^\fz \frac{\,dt}{t^{2(\wz M-k)+1}}\\
&&\hs\hs + \sup_{0\le k\le M}\sum_{j=3}^\fz
(2^jr_B)^{\frac {n+\ez}2+N(\frac 1{p_\oz}-1)}
\|b\|_{L^2(\cx)}\int_R^\fz \exp\lf\{-\frac{\dist(B,U_j(B))^2}{t^2}\r\}
\frac{\,dt}{t^{2(\wz M-k)+1}}\\
&&\hs\ls R^{-2(\wz M-M)}+\sum_{j=3}^\fz (2^jr_B)^{\frac {n+\ez}2+N(\frac 1{p_\oz}-1)}
\int_R^\fz \lf(\frac{t}{2^jr_B}\r)^{n/2+\ez+
N(\frac 1{p_\oz}-1)}\frac{\,dt}{t^{2(\wz M-M)+1}}\\
&&\hs\ls R^{-\ez}\to 0,
\end{eqnarray*}
as $R\to\fz$.

Similarly, the term $\mathrm{H}$ is controlled by
\begin{eqnarray*}
&&\lf\{\int_{\cx}\frac{|(I-(I+L)^{-1})^Mf(x)|^2}
{r_B+[d(x,x_B)]^{n+\ez+2N(\frac 1{p_\oz}-1)}}\,d\mu(x)\r\}^{1/2}
\sup_{0\le k\le M}\lf\{\int_\cx\bigg|\int_0^\dz (t^2L)^{M+1} \r. \\
&&\hs\hs\times \lf.  e^{-2t^2L} L^{\wz M-k}b(x)\frac{\,dt}{t}
\bigg|^2(r_B+[d(x,x_B)]^{n+\ez+2N(\frac 1{p_\oz}-1)})\,d\mu(x)\r\}^{1/2}\\
&&\hs\ls \sum_{j=0}^\fz \sup_{0\le k\le M}(2^jr_B)^{(n+\ez)/2+N(\frac 1{p_\oz}-1)}\\
&&\hs\hs \times\lf\{\int_{U_j(B)}\bigg|\int_0^\dz (t^2L)^{M+1}e^{-2t^2L}
L^{\wz M-k}b(x)\frac{\,dt}{t}\bigg|^2\,d\mu(x)\r\}^{1/2}
\sim \sum_{j=0}^\fz \mathrm{H}_j.
\end{eqnarray*}
For $j\ge3$, we further have
\begin{eqnarray}\label{4.9}
\quad\quad\mathrm{H}_j&&\ls \sup_{0\le k\le M} (2^jr_B)^{\frac {n+\ez}2
+N(\frac 1{p_\oz}-1)}
\int_0^\dz \|(t^2L)^{M+1}e^{-2t^2L}
L^{\wz M-k}b\|_{L^2(U_j(B))}\frac{\,dt}{t}\\
&&\ls \sup_{0\le k\le M} (2^jr_B)^{\frac {n+\ez}2+N(\frac 1{p_\oz}-1)}
\|L^{\wz M-k}b\|_{L^2(\cx)}\int_0^\dz \exp\lf\{-\frac{\dist(B,U_j(B))^2}{t^2}\r\}
\frac{\,dt}{t}\nonumber\\
&&\ls (2^jr_B)^{\frac {n+\ez}2+N(\frac 1{p_\oz}-1)} \int_0^\dz
\lf(\frac{t}{2^jr_B}\r)^{n/2+\ez+N(\frac 1{p_\oz}-1)}\frac{\,dt}{t}
\ls 2^{-j\ez/2}\dz^{n/2+\ez+N(\frac 1{p_\oz}-1)}.\nonumber
\end{eqnarray}
Notice that for each $i\in\cn$,
$(\dz^2L)^ie^{-2\dz^2L}\to 0$ and $e^{-2\dz^2L}-I\to 0$ in the
strong operator topology as $\dz\to 0$. Thus, for $j=0,1,2$, we
obtain
\begin{eqnarray}\label{4.10}
\mathrm{H}_j&&\ls \sup_{0\le k\le M}\lf\|\int_0^\dz
(t^2L)^{(M+1)}e^{-2t^2L} L^{\wz
M-k}b(x)\frac{\,dt}{t}\r\|_{L^2(U_j(B))}\\
&&\ls \sup_{0\le k\le M} \bigg[\sum_{i=1}^M
\|(\dz^2L)^ie^{-2\dz^2L}(L^{\wz
M-k}b)\|_{L^2(U_j(B))}\nonumber\\
&&\hs+\|(e^{-2\dz^2L}-I)(L^{\wz M-k}b)\|_{L^2(U_j(B))}
\bigg]\nonumber\to 0,\nonumber
\end{eqnarray}
as $\dz\to 0$. The estimates \eqref{4.9} and \eqref{4.10}
imply that $\mathrm{H}\to 0$ as $\dz\to 0$, and hence complete the
proof of Proposition \ref{p4.6}.
\end{proof}

To prove that Proposition \ref{p4.6} holds for all $f\in
\bmoo$, we need the following ``dyadic cubes" on spaces of homogeneous type
constructed by Christ \cite[Theorem 11]{cr}.

\begin{lem}\label{l4.1}
There exists a collection $\{Q^k_\az\subset \cx:\,k\in\zz,\, \az\in
I_k\}$ of open subsets, where $I_k$ denotes some (possibly finite)
index set depending on $k$, and constants $\dz\in(0, 1)$, $a_0\in(0,
1)$ and $C_5\in (0,\fz)$ such that

(i) $\mu(\cx\setminus \cup_{\az}Q_\az^k)=0$ for all $k\in\zz $;

(ii) if $i\ge k$, then either $Q^i_\az\subset Q^k_\bz$ or
$Q^i_\az\cap Q^k_\bz=\emptyset$;

(iii) for each $(k,\az)$ and each $i < k$, there exists a unique $\bz$ such that
$Q^k_\az\subset Q^i_\bz$;

(iv) the diameter of $Q_\az^k$ is no more than $C_5 \dz^k$;

(v) each $Q_\az^k$ contains certain ball $B(z^k_\az, a_0\dz^k)$.
\end{lem}

From Proposition \ref{p4.6} and Lemma \ref{l4.1},
we deduce the following conclusion.

\begin{cor}\label{c4.3}
Let $L$, $\oz$, $\ro$ and $M$ be as in Definition \ref{d4.4} and
$\wz M>M+\frac n4+ \frac N 2 (\frac 1 {p_\oz}-1)$. Then for all
$(\oz,\wz M)$-atom $\az$ and $f\in \mathrm{BMO}^{M}_{\ro,L}(\cx)$,
\begin{equation*}
\la f, \az\ra =\wz C_M\int_{\cx\times(0,\fz)} (t^2L)^Me^{-t^2L}f(x)
\overline{t^2Le^{-t^2L}\az(x)}\frac{\,d\mu(x)\,dt}{t},
\end{equation*}
where $\wz C_M$ is as in Proposition \ref{p4.6}.
\end{cor}

\begin{proof}
Let $\ez\in (0,\wz M-M-\frac n4- \frac N 2 (\frac 1 {p_\oz}-1))$. By
Proposition \ref{p4.6}, we only need to show that \eqref{4.8} with
such an $\ez$ holds for all $f\in \mathrm{BMO}^{M}_{\ro,L}(\cx)$.

Let all the notation be the same as in Lemma \ref{l4.1}. For each
$j\in\zz$, choose $k_j\in\zz$ such that $C_5\dz^{k_j}\le
2^j<C_5\dz^{k_j-1}$. Let $B=B(x_0,1)$, where $x_0$ is as in
\eqref{4.8}, and
$$M_j \equiv \{\bz\in I_{k_0}:\ Q^{k_0}_\bz\cap  B(x_0, C_5
\dz^{k_j-1})\neq \emptyset\}.$$
Then for each $j\in \zz_+$,
\begin{equation}\label{4.11}
U_j(B)\subset B(x_0,C_5\dz^{k_j-1}) \subset \bigcup_{\bz\in M_j}Q_\bz^{k_0}
\subset B(x_0,2C_5\dz^{k_j-1}).
\end{equation}
By Lemma \ref{l4.1}, the sets $Q_\bz^{k_0}$ for all $\bz\in M_j$ are disjoint.
Moreover, by (iv) and (v) of Lemma \ref{l4.1},
there exists $z_\bz^{k_0}\in Q_\bz^{k_0}$ such that
\begin{equation}\label{4.12}
B(z_\bz^{k_0},a_0\dz^{k_0})\subset Q_\bz^{k_0}
\subset B(z_\bz^{k_0},C_5\dz^{k_0})\subset
B(z_\bz^{k_0},1).
\end{equation}
Thus, by Proposition \ref{p4.4}, \eqref{2.4} and the fact
that $\ro$ is of upper type $1/p_\oz-1$, we have
\begin{eqnarray*}
\mathrm{H}&&\equiv\lf\{\int_{\cx}\frac{|(I-(I+L)^{-1})^Mf(x)|^2}
{1+[d(x,x_0)]^{n+\ez+2N(1/p_\oz-1)}}\,d\mu(x)\r\}^{1/2}\\
&&\ls \sum_{j=0}^\fz 2^{-j[(n+\ez)/2+N(1/p_\oz-1)]}
\lf\{\sum_{\bz\in M_j}\int_{Q_\bz^{k_0}}|(I-(I+L)^{-1})^Mf(x)|^2\,d\mu(x)\r\}^{1/2}\\
&&\ls \sum_{j=0}^\fz 2^{-j[(n+\ez)/2+N(1/p_\oz-1)]}
\lf\{\sum_{\bz\in M_j}[\ro(V(z_\bz^{k_0},1))]^2V(z_\bz^{k_0},1)
\|f\|^2_{\bmoo}\r\}^{1/2}\\
&&\ls \sum_{j=0}^\fz 2^{-j[(n+\ez)/2+N(1/p_\oz-1)]}
\lf\{\sum_{\bz\in M_j}2^{2jN(1/p_\oz-1)}[\ro(V(x_0,1))]^2V(z_\bz^{k_0},1)\r\}^{1/2}\\
&&\ls \sum_{j=0}^\fz 2^{-j(n+\ez)/2}
\lf\{\sum_{\bz\in M_j}V(z_\bz^{k_0},1)\r\}^{1/2}.
\end{eqnarray*}
By \eqref{4.11}, \eqref{4.12} and \eqref{2.3}, we obtain
$$\sum_{\bz\in M_j}V(z_\bz^{k_0},1)\ls \sum_{\bz\in M_j} V(z_\bz^{k_0},a\dz^{k_0})
\ls\sum_{\bz\in M_j} V(Q_\bz^{k_0})\ls V(x_0, 2C_5\dz^{k_j-1})\ls
2^{jn}V(B),$$ which further implies that $\mathrm{H}<\fz$, and hence
completes the proof of Corollary \ref{c4.3}.
\end{proof}

Using Corollary \ref{c4.3}, we now establish the duality between
$H_{\oz, L}(\cx)$ and $\mathrm{BMO}^{M}_{\ro,\,L}(\cx)$.

\begin{thm}\label{t4.1}
Let $L$ satisfy Assumptions (A) and (B), $\oz$ satisfy Assumption (C),
$\ro$ be as in \eqref{2.10}, $M> \frac n2(\frac{1}{p_\oz}-\frac 12)$ and
$\wz M>M+\frac n4+ \frac N 2(\frac 1{ p_\oz}-1)$.
Then $(H_{\oz,L}(\cx))^\ast$, the
dual space of $H_{\oz,L}(\cx)$, coincides with
$\mathrm{BMO}^{M}_{\ro,\,L}(\cx)$
in the following sense.

(i) Let  $g\in \mathrm{BMO}^{M}_{\ro,\,L}(\cx)$. Then the linear
functional $\ell_g$, which is initially defined on $H^{{\rm at},\wz
M}_{\oz,\rm fin}(\cx)$ by
\begin{equation}\label{4.13}
\ell_g(f)\equiv \la g, f\ra,
\end{equation}
has a unique extension to $H_{\oz,L}(\cx)$ with
$\|\ell_g\|_{(H_{\oz,L}(\cx))^\ast}\le
C\|g\|_{\mathrm{BMO}^{M}_{\ro,L}(\cx)},$ where $C$ is a positive
constant independent of $g.$

(ii) Conversely, let $\ez>n(1/p_\oz-1/{p_\oz^+})$. Then for any
$\ell\in (H_{\oz,L}(\cx))^\ast$, $\ell\in
\mathrm{BMO}^{M}_{\ro,\,L}(\cx)$ with
$\|\ell\|_{\mathrm{BMO}^{M}_{\ro,\,L}(\cx)}\le
C\|\ell\|_{(H_{\oz,L}(\cx))^\ast}$ and \eqref{4.13} holds for all
$f\in H^{{\rm mol},M,\ez}_{\oz,\rm fin}(\cx)$, where $C$ is a
positive constant independent of $\ell.$
\end{thm}

\begin{proof}
Let $g\in \mathrm{BMO}^{M}_{\ro,L}(\cx)$. For any $f\in H^{{\rm
at},\wz M}_{\oz,\rm fin}(\cx)$, by Proposition \ref{p4.1}, we have
that $t^2Le^{-t^2L}f\in T_{\oz}(\cx)$. By Theorem \ref{t3.1}, there
exist $\{\lz_j\}_{j=1}^\fz\subset \cc$ and $T_\oz(\cx)$-atoms
$\{a_j\}_{j=1}^\fz$ supported in $\{\widehat B_j\}_{j=1}^\fz$ such
that \eqref{3.1} and \eqref{3.2} hold. This, together with Corollary
\ref{c4.3}, the H\"older inequality, Proposition \ref{p4.5} and
Remark \ref{r3.1}(ii), yields that
\begin{eqnarray}\label{4.14}
|\la g, f \ra|&&=\bigg|C_{\wz M}\int_0^\fz\int_\cx(t^2L)^Me^{-t^2L}g(x)
\overline{t^2Le^{-t^2L}f(x)}\frac{\,d\mu(x)\,dt}{t}\bigg|\\
&&\ls\sum_j|\lz_j|\int_0^\fz\int_\cx|(t^2L)^Me^{-t^2L}g(x)
a_j(x,t)|\frac{\,d\mu(x)\,dt}{t}\nonumber\\
&&\ls\sum_j|\lz_j|\|a_j\|_{T^2_2(\cx)}
\lf(\int_{\widehat B_j}|(t^2L)^Me^{-t^2L}g(x)|^2
\frac{\,d\mu(x)\,dt}{t}\r)^{1/2}\nonumber\\
&&\ls\sum_j|\lz_j|\|g\|_{\bmoo}\ls
\|t^2Le^{-t^2L}f\|_{T_\oz(\cx)}\|g\|_{\bmoo}\nonumber\\
&&\sim \|f\|_{H_{\oz,L}(\cx)}\|g\|_{\bmoo}\nonumber.
\end{eqnarray}
Then by Proposition \ref{c4.2}, we obtain (i).

Conversely, let $\ell\in (H_{\oz,L}(\cx))^\ast$. If $g\in
\cm_\oz^{M,\ez}(L)$, then $g$ is a multiple of an
$(\oz,M,\ez)$-molecule. Moreover, if $\ez>n(1/p_\oz-1/p_\oz^+)$,
then by Proposition \ref{p4.1}, we have $g\in H_{\oz,L}(\cx)$ and
hence $\cm_\oz^{M,\ez}(L)\subset H_{\oz,L}(\cx)$. Then
$\ell\in\cm_\oz^M(\cx)$.

On the other hand, for any ball $B$, let $\phi\in L^2(B)$ with
$\|\phi\|_{L^2(B)}\le \frac{1}{[V(B)]^{1/2}\ro(V(B))}$ and $\wz
\bz\equiv (I-[I+r_B^2L]^{-1})^M\phi$. Obviously, $\wz
\bz=(r_B^2L)^M(I+r_B^2L)^{-M}\phi\equiv L^M \wz b$. Then from the
fact that $\{(t^2L)^k(I+r_B^2L)^{-M}\}_{0\le k\le M}$ satisfy the
Davies-Gaffney estimate (see Lemma \ref{l2.1}), we deduce that for
each $j\in \zz_+$ and $k =0,1,\,\cdots,M$, we have
\begin{eqnarray*}\|(r_B^2L)^k\wz b\|_{L^2(U_j(B))}&&=r_B^{2M}
\|(I-[I+r_B^2L]^{-1})^k(I+r_B^2L)^{-(M-k)}\phi\|_{L^2(U_j(B))}\\
&&\ls r_B^{2M}\exp\lf\{-\frac{\dist(B,U_j(B))}{r_B}\r\}\|\phi\|_{L^2(B)}\\
&&\ls r_B^{2M}2^{-2j(M+\ez)}2^{jn(1/p_\oz-1/2)}[V(2^jB)]^{-1/2}[\ro(V(2^jB))]^{-1}\\
&&\ls r_B^{2M}2^{-2j\ez}[V(2^jB)]^{-1/2}[\ro(V(2^jB))]^{-1},
\end{eqnarray*}
since $2M>n(1/p_\oz-1/2)$. Thus, $\wz \bz$ is a multiple of an
$(\oz,M,\ez)$-molecule. Since $(I-[I+t^2L]^{-1})^M\ell$ is well
defined and belongs to $L^2_{\loc}(\cx)$ for every $t>0$, by the
fact that $\|\wz\bz\|_{H_{\oz,L}(\cx)}\ls 1$, we have
\begin{eqnarray*}
|\la (I-[I+r_B^2L]^{-1})^M\ell,\phi\ra|=|\la \ell,(I-[I+r_B^2L]^{-1})^M\phi\ra|
=|\la \ell,\wz \bz\ra|\ls \|\ell\|_{(H_{\oz,L}(\cx))^\ast},
\end{eqnarray*}
which further implies that
\begin{eqnarray*}
&&\frac{1}{\ro(V(B))}\lf(\frac{1}{V(B)}\int_{B}
|(I-[I+r_B^2L]^{-1})^M\ell(x)|^2\,d\mu(x)\r)^{1/2}\\
&&\hs=\sup_{\|\phi\|_{L^2(B)}\le 1}\lf|\bigg\la \ell,(I-[I+r_B^2L]^{-1})^M
\frac{\phi}{[V(B)]^{1/2}\ro(V(B))}\bigg\ra\r|
\ls\|\ell\|_{(H_{\oz,L}(\cx))^\ast}.
\end{eqnarray*}
Thus, by Proposition \ref{p4.4},
we obtain $\ell\in \bmoo$, which completes the proof of Theorem \ref{t4.1}.
\end{proof}

\begin{rem}\label{r4.2}\rm
By Theorem \ref{t4.1}, we have that for all $M> \frac n2(\frac{1}{p_\oz}-\frac 12)$,
the spaces $\bmoo$ coincide with equivalent norms;
thus, in what follows, we denote $\bmoo$ simply by $\bmo$.
\end{rem}

Recall that  a measure $d\mu$ on $\cx\times (0,\fz)$ is
called a $\ro$-Carleson measure if
$$\|d\mu\|_\ro\equiv\sup_{B\subset
\cx}\lf\{\frac{1}{V(B)[\ro(V(B))]^2}\int_{\widehat{B}}\,|d\mu|\r\}^{1/2}<\fz,$$
where the supremum is taken over all balls $B$ of $\cx$ and
$\widehat{B}$ denotes the tent over $B$; see \cite{hsv}.

Using Theorem \ref{t4.1} and Proposition \ref{p4.5}, we obtain
the following $\ro$-Carleson measure characterization of $\bmo$.

\begin{thm}\label{t4.2}
Let $L$ satisfy Assumptions (A) and (B), $\oz$ satisfy Assumption (C),
$\ro$ be as in \eqref{2.10} and $M>\frac n2(\frac{1}{p_\oz}-\frac 12)$.
Then the following conditions are equivalent:

(i) $f\in \bmo$;

(ii) $f\in \cm_\oz^M(\cx)$ satisfies \eqref{4.8} for some $\ez>0$,
and $d\mu_f$ is a $\ro$-Carleson measure, where $d\mu_f$ is defined
by $d\mu_f\equiv|(t^2L)^Me^{-t^2L}f(x)|^2
\frac{\,d\mu(x)\,dt}{t}.$

Moreover, $\|f\|_{\bmo}$ and $\|d\mu_f\|_\ro$ are comparable.
\end{thm}

\begin{proof}
It follows from Proposition \ref{p4.5} and the proof of Corollary
\ref{c4.3} that (i) implies (ii).

Conversely, let $\wz M>M+\ez+\frac n4+ \frac N 2(\frac 1{p_\oz}-1)$.
By Proposition \ref{p4.6}, we have
\begin{equation*}
\la f, g\ra =\wz C_M\int_{\cx\times(0,\fz)} (t^2L)^Me^{-t^2L}f(x)
\overline{t^2Le^{-t^2L}g(x)}\frac{\,d\mu(x)\,dt}{t},
\end{equation*}
where $g$ is any finite combination of $(\oz,\wz M)$-atoms. Then
similarly to the estimate of \eqref{4.14}, we obtain that
\begin{equation}\label{4.15}
|\la f, g\ra|\ls \|d\mu_f\|_{\ro}\|g\|_{H_{\oz,L}(\cx)}.
\end{equation}
Since, by Corollary \ref{c4.2}, $H_{\oz,\rm fin}^{{\rm at},\wz
M}(\cx)$ is dense in $H_{\oz,L}(\cx)$, this together with Theorem
\ref{t4.1} and \eqref{4.15} implies that $f\in
(H_{\oz,L}(\cx))^\ast=\bmo$, which completes the proof of Theorem
\ref{t4.2}.
\end{proof}

\section{Characterizations of Orlicz-Hardy spaces \label{s5}}

\hskip\parindent In this section, we characterize  the Orlicz-Hardy spaces
by atoms, molecules and the Lusin-area function
associated with the Poisson semigroup.
Let us begin with some notions.

\begin{defn}\label{d5.1} Let $L$ satisfy Assumptions (A) and (B),
$\oz$ satisfy Assumption (C) and $M>\frac n2(\frac{1}{p_\oz}-\frac
12)$. A distribution $f\in (\bmo)^\ast$ is said to be in the space
$H_{\oz,{\rm at}}^M(\cx)$ if there exist $\{\lz_j\}_{j=1}^\fz\subset
\cc$ and $(\oz,M)-$atoms $\{\az_j\}_{j=1}^\fz$ such that
$f=\sum_{j=1}^\fz\lz_j\az_j$ in the norm of $(\bmo)^\ast$ and
$\sum_{j=1}^\fz V(B_j)\oz(\frac{|\lz_j|}{V(B_j)\ro(V(B_j))})<\fz,$
where for each $j$, $\supp \az_j\subset B_j$.

Moreover, for any $f\in H_{\oz,{\rm at}}^M(\cx)$, its norm is defined by
$\|f\|_{H_{\oz,{\rm at}}^M(\cx)}\equiv \inf \Lambda(\{\lz_j\az_j\}_j),$
where $\Lambda(\{\lz_j\az_j\}_j)$ is the same as in Proposition
\ref{p4.3} and the infimum is taken over all possible decompositions
of $f$.
\end{defn}

\begin{defn}\label{d5.2} Let $L$ satisfy Assumptions (A) and (B),
$\oz$ satisfy Assumption (C), $M>\frac n2(\frac{1}{p_\oz}-\frac 12)$
and $\ez>n(1/p_\oz-1/{p_\oz^+})$. A distribution $f\in (\bmo)^\ast$
is said to be in the space $H_{\oz,{\rm mol}}^{M,\ez}(\cx)$ if there
exist $\{\lz_j\}_{j=1}^\fz\subset \cc$ and $(\oz,M,\ez)-$molecules
$\{\bz_j\}_{j=1}^\fz$ such that $f=\sum_{j=1}^\fz\lz_j\bz_j$ in the
norm of $(\bmo)^\ast$ and
$\sum_{j=1}^\fz V(B_j)\oz(\frac{|\lz_j|}{V(B_j)\ro(V(B_j))})<\fz,$
where for each $j$, $\bz_j$ is associated to the ball $B_j$.

Moreover, for any $f\in H_{\oz,{\rm mol}}^{M,\ez}(\cx)$, its norm is defined by
$\|f\|_{H_{\oz,{\rm mol}}^{M,\ez}(\cx)}\equiv \inf \Lambda(\{\lz_j\bz_j\}_j),$
where $\Lambda(\{\lz_j\bz_j\}_j)$ is the same as in Proposition
\ref{p4.3} and the infimum is taken over all possible decompositions
of $f$.
\end{defn}

For all $f\in L^2(\cx)$ and $x\in\cx$, define the Lusin area function associated
to the Poisson semigroup by
\begin{equation}\label{5.1}
\cs_Pf(x)\equiv \bigg(\iint_{\Gamma(x)}|t\sqrt L e^{-t\sqrt L}f(y)|^2
\frac{\,d\mu(y)}{V(x,t)}\frac{\,dt}{t}\bigg)^{1/2}.
\end{equation}

Similarly to Definition \ref{d4.1}, we define the space
$H_{\oz,\cs_P}(\cx)$ as follows.

\begin{defn}\label{d5.3}
Let $L$ satisfy Assumptions (A) and (B), $\oz$ satisfy Assumption
(C) and $\overline{\car(L)}$ be as in \eqref{4.2}. A function $f\in
\overline {\car(L)}$ is said to be in $\wz H_{\oz,\,\cs_P}(\cx)$ if
$\cs_P(f)\in L(\oz)$; moreover, define
$$\|f\|_{H_{\oz,\cs_P}(\cx)}\equiv \|\cs_P(f)\|_{L(\oz)}
=\inf\lf\{\lz>0:\int_{\cx}\oz\lf
(\frac{\cs_P(f)(x)}{\lz}\r)\,d\mu(x)\le 1\r\}.$$ The Orlicz-Hardy space
$H_{\oz,\cs_P}(\cx)$ is defined to be the  completion of $\wz
H_{\oz,\cs_P}(\cx)$  in the norm $\|\cdot\|_{H_{\oz,\cs_P}(\cx)}.$
\end{defn}

We now show that the spaces $H_{\oz,L}(\cx)$, $H_{\oz,{\rm at}}^M(\cx)$,
$H_{\oz,{\rm mol}}^{M,\ez}(\cx)$ and $H_{\oz,\cs_P}(\cx)$
coincide with equivalent norms.

\subsection{Atomic and molecular characterizations\label{s5.1}}

\hskip\parindent In this subsection, we establish the atomic and
the molecular characterizations
of the Orlicz-Hardy spaces. We start with the following auxiliary result.

\begin{prop}\label{p5.1} Let $L$ satisfy Assumptions (A) and (B) and
$\oz$ satisfy Assumption (C).
Fix $t\in (0,\fz)$ and $\wz B\equiv B(x_0,R)$. Then there exists a positive
constant $C(t,R,\wz B)$ such that for all $\phi\in L^2(\wz B)$, $t^2L e^{-t^2L}
\phi\in\bmo$ and
$$\|t^2L e^{-t^2L}\phi\|_{\bmo}\le C(t,R,\wz B)\|\phi\|_{L^2(\wz B)}.$$
\end{prop}

\begin{proof} Let $M> \frac n2(\frac{1}{p_\oz}-\frac 12)$.
For any ball $B\equiv B(x_B,r_B)$, let
$$\mathrm{H}\equiv \frac{1}{\ro(V(B))}
\lf(\frac{1}{V(B)}\int_{B}|(I-e^{-r_B^2L})^Mt^2L
e^{-t^2L}\phi(x)|^2\,d\mu(x)\r)^{1/2}.$$

We now consider two cases. {\emph{Case i)}} $r_B\ge R$. In this
case, either $\wz B\subset 2^{3}B$ or there exists $k\ge 3$ such
that $\wz B\subset (2^{k+1}B\setminus 2^{k-1}B)$. If $\wz B\subset
2^{3}B$, we have $[V(\wz B)]^{1/2}\ro(V(\wz B))\ls
[V(B)]^{1/2}\ro(V(B))$. This together with the
$L^2(\cx)$-boundedness of the operator $t^2Le^{-t^2L}$ (see Lemma
\ref{l2.1}) yields that
$\mathrm{H}\ls\frac{\|\phi\|_{L^2(\wz B)}}{[V(\wz B)]^{1/2}\ro(V(\wz B))},$
which is a desired estimate.

If $B\subset 2^{k+1}B\setminus 2^{k-1}B$ for some $k\ge 3$, we then
have $\wz B\subset 2^{k+1}B$ and $\dist(\wz B,B)\ge 2^{k-2}r_B\ge
2^{k-2}R$. Thus, by the Davies-Gaffney estimate, we obtain
\begin{eqnarray*}
\mathrm{H}&&\ls \frac{2^{n(k+2)(1/p_\oz-1/2)}}{[V(2^{k+2}B)]^{1/2}\ro(V(2^{k+2}B))}
\exp\lf\{-\frac{\dist(\wz B,B)^2}{t^2}\r\}\|\phi\|_{L^2(\wz B)}\\
&&\ls\frac{2^{nk(1/p_\oz-1/2)}\|\phi\|_{L^2(\wz B)}}{[V(\wz B)]^{1/2}\ro(V(\wz B))}
\lf(\frac{t}{2^{k}R}\r)^{n(1/p_\oz-1/2)}\ls
\lf(\frac{t}{R}\r)^{n(1/p_\oz-1/2)}\frac{\|\phi\|_{L^2(\wz B)}}
{[V(\wz B)]^{1/2}\ro(V(\wz B))},
\end{eqnarray*}
which is also a desired estimate.

{\emph{Case ii)}} $r_B<R$. In this case, we further consider two
subcases. If $d(x_B,x_0)\le 4R$, then $\wz B\subset B(x_B,5R)$ and
hence
\begin{equation}\label{5.2}
[V(\wz B)]^{1/2}\ro(V(\wz B))\ls \lf(\frac{R}{r_B}\r)^{n(1/p_\oz-1/2)}
[V(B)]^{1/2}\ro(V(B)).
\end{equation}
On the other hand, noticing that $I-e^{-r^2_BL}=\int_0^{r_B^2}Le^{-rL}\,dr$, by
the Minkowski inequality and the $L^2(\cx)$-boundedness of the
operator $t^2Le^{-t^2L}$, we have
\begin{eqnarray}\label{5.3}
\quad\quad&& \lf(\int_{B}|(I-e^{-r_B^2L})^Mt^2L e^{-t^2L}
\phi(x)|^2\,d\mu(x)\r)^{1/2}\\
&&\hs=\lf(\int_{B}\lf|\int_0^{r_B^2}\cdots\int_0^{r_B^2} t^2L^{M+1}
e^{-(r_1+\cdots+r_M+t^2)L}\phi(x)\,dr_1\cdots\,dr_M\r|^2\,d\mu(x)\r)^{1/2}\nonumber\\
&&\hs\ls\|\phi\|_{L^2(\wz B)}\int_0^{r_B^2}\cdots\int_0^{r_B^2}
\frac{t^2}{(r_1+\cdots+r_M+t^2)^{M+1}}\,dr_1\cdots\,dr_M
\ls \lf(\dfrac{r_B}{t}\r)^{2M}\|\phi\|_{L^2(\wz B)}.\nonumber
\end{eqnarray}
By $2M>n(\frac{1}{ p_\oz}-\frac 12)$ and
the estimates \eqref{5.2} and \eqref{5.3},
 we obtain
$$\mathrm{H}\ls \lf(\frac{R}{t}\r)^{2M}\frac{\|\phi\|_{L^2(\wz B)}}
{[V(\wz B)]^{1/2}\ro(V(\wz B))},$$
which is also a desired estimate.

If $d(x_B,x_0)> 4R$, then there exists $k_0\ge 3$ such
that $d(x_B,x_0)\in (2^{k_0-1}R,2^{k_0}R]$.
Thus, $\wz B\subset B(x_B,2^{k_0+1}R)$, $\dist(\wz B,B)\ge 2^{k_0-2}R$ and
\begin{equation}\label{5.4}
[V(\wz B)]^{1/2}\ro(V(\wz B))\ls
\lf(\frac{2^{k_0}R}{r_B}\r)^{n(1/p_\oz-1/2)}[V(B)]^{1/2}\ro(V(B)).
\end{equation}
By the Davies-Gaffney estimate, we further obtain
\begin{eqnarray*}
&& \lf(\int_{B}|(I-e^{-r_B^2L})^Mt^2L e^{-t^2L}\phi(x)|^2\,d\mu(x)\r)^{1/2}\\
&&\hs=\lf(\int_{B}\lf|\int_0^{r_B^2}\cdots\int_0^{r_B^2} t^2L^{M+1}
e^{-(r_1+\cdots+r_M+t^2)L}\phi(x)\,dr_1\cdots\,dr_M\r|^2\,d\mu(x)\r)^{1/2}\nonumber\\
&&\hs\ls\int_0^{r_B^2}\cdots\int_0^{r_B^2}\frac{t^2}{(r_1+\cdots+r_M+t^2)^{M+1}}\\
&&\hs\hs \times \exp\lf\{-\frac{\dist(\wz B,B)^2}
{r_1+\cdots+r_M+t^2}\r\}\|\phi\|_{L^2(\wz B)}
\,dr_1\cdots\,dr_M\nonumber\\
&&\hs\ls \lf(\frac{r_B}{t}\r)^{2M}\lf(\frac{t+r_B}{2^{k_0}R}\r)^{n(1/p_\oz-1/2)}
\|\phi\|_{L^2(\wz B)},
\end{eqnarray*}
which, together with \eqref{5.4}, $r_B<R$ and $2M>n(\frac{1}{ p_\oz}-\frac 12)$,
yields that
$$\mathrm{H}\ls\lf(\frac{R+t}{t}\r)^{2M}\frac{\|\phi\|_{L^2(\wz B)}}
{[V(\wz B)]^{1/2}\ro(V(\wz B))}.$$
This is also a desired estimate, which completes
the proof of Proposition \ref{p5.1}.
\end{proof}

From Proposition \ref{p5.1}, it follows that for each $f\in (\bmo)^\ast$,
$t^2Le^{-t^2L}f$ is well defined. In fact, for any ball $B\equiv B(x_B,r_B)$ and
$\phi\in L^2(B)$, by Proposition \ref{p5.1}, we have
$$|\la t^2Le^{-t^2L}f,\phi\ra|\equiv|\la f,t^2Le^{-t^2L}\phi\ra|
\le C(t, r_B, B)\|\phi\|_{L^2(B)}\|f\|_{(\bmo)^\ast},$$ which
implies that $t^2Le^{-t^2L}f\in L^2_{\loc}(\cx)$ in the sense of
distributions. Moreover, recalling that the atomic decomposition
obtained in Corollary \ref{c4.1} holds in $H_{\oz,L}(\cx)$, then by
Theorem \ref{t4.1}, we see the atomic decomposition also holds in
$(\bmo)^\ast$. Applying these observations, similarly to the proof
of \cite[Theorem 5.1]{jy}, we have the following result. We omit the
details here.

\begin{thm}\label{t5.1}
Let $L$ satisfy Assumptions (A) and (B), $\oz$ satisfy Assumption (C),
$M>\frac n2(\frac{1}{ p_\oz}-\frac 12)$  and $\ez>n(1/p_\oz-1/{p_\oz^+})$.
Then the spaces $H_{\oz,L}(\cx)$, $H^{M}_{\oz,{\rm at}}(\cx)$
and $H^{M,\ez}_{\oz,{\rm mol}}(\cx)$
coincide with equivalent norms.
\end{thm}

\subsection{A characterization by the Lusin area function $\cs_P$\label{s5.2}}

\hskip\parindent In this subsection, we characterize the space $H_{\oz,L}(\cx)$
by the Lusin area function $\cs_P$ as in \eqref{5.1}. We start with the
following auxiliary conclusion.

\begin{lem}\label{l5.1}
Let $K\in \zz_+$. Then the operator $(t\sqrt L)^{K}e^{-t\sqrt L}$
is bounded on $L^2(\cx)$ uniformly in $t$. Moreover, there exists a
positive constant $C$ such that
for all closed sets $E,\,F$ in $\cx$ with $\dist(E,F)>0$, all $t>0$ and $f\in L^2(E)$,
$$\|(t\sqrt L)^{2K}e^{-t\sqrt L}f\|_{L^2(F)}+
\|(t\sqrt L)^{2K+1}e^{-t\sqrt L}f\|_{L^2(F)}\le C
\lf(\frac{t}{\dist(E,F)}\r)^{2K+1}\|f\|_{L^2(E)}.$$
\end{lem}

\begin{proof}
It was proved in \cite[Lemma 5.1]{hm1} and \cite[Lemma 4.15]{hlmmy} that
$$\|(t\sqrt L)^{2K}e^{-t\sqrt L}f\|_{L^2(F)}\ls
\lf(\frac{t}{\dist(E,F)}\r)^{2K+1}\|f\|_{L^2(E)}.$$

To establish a similar estimate for $(t\sqrt L)^{2K+1}e^{-t\sqrt L}f$,
we first notice that the subordination formula
\begin{equation}\label{5.5}
e^{-t\sqrt L}f=\frac{1}{\sqrt \pi}\int_0^\fz \frac{e^{-u}}{\sqrt u}
e^{-\frac{t^2}{4u}L}f\,du
\end{equation}
implies that
$$\sqrt Le^{-t\sqrt L}f=-\frac{\partial}{\partial t}e^{-t\sqrt L}f=
\frac{1}{2\sqrt \pi}\int_0^\fz \frac{te^{-u}L}{u^{3/2}}e^{-\frac{t^2}{4u}L}f\,du.$$
Then by the Minkowski inequality and Lemma \ref{l2.1}, we obtain
\begin{eqnarray*}
\|(t\sqrt L)^{2K+1}e^{-t\sqrt L}f\|_{L^2(F)}&&\ls \int_0^\fz
\frac{e^{-u}}{\sqrt u}\lf\|\bigg(\frac{t^2}{4u}L\bigg)^{K+1}
e^{-\frac{t^2}{4u}L} f\r\|_{L^2(F)}u^K\,du\\
&&\ls\int_0^\fz \frac{e^{-u}}{\sqrt u} \exp
\lf\{-\frac{u\dist(E,F)^2}{C_3t^2}\r\}u^K\,du \|f\|_{L^2(E)}\\
&&\ls\lf(\frac{t}{\dist(E,F)}\r)^{2K+1}\|f\|_{L^2(E)}.
\end{eqnarray*}

To show that $(t\sqrt L)^{2K+1}e^{-t\sqrt L}$ is bounded on $L^2(\cx)$
uniformly in $t$, by the boundedness of $t^2Le^{-t^2L}$ on $L^2(\cx)$
uniformly in $t$, we have
\begin{eqnarray*}
\|(t\sqrt L)^{2K+1}e^{-t\sqrt L}f\|_{L^2(\cx)}&&\ls \int_0^\fz
\frac{e^{-u}}{\sqrt u}\lf\|\bigg(\frac{t^2}{4u}\bigg)^{K+1}e^{-\frac{t^2}{4u}L}
f\r\|_{L^2(\cx)}u^K\,du\\
&&\ls\int_0^\fz \frac{e^{-u}}{\sqrt u} u^K\,du\|f\|_{L^2(\cx)}\ls \|f\|_{L^2(\cx)}.
\end{eqnarray*}
Similarly, we have that $(t\sqrt L)^{2K}e^{-t\sqrt L}$ is bounded on $L^2(\cx)$
uniformly in $t$, which completes the proof of Lemma \ref{l5.1}.
\end{proof}

Similarly to \cite[Lemma 5.1]{jy}, we have the following lemma. We
omit the details. Recall that a nonnegative sublinear operator $T$
means that $T$ is sublinear and $Tf\ge 0$ for all $f$ in the domain
of $T$.
\begin{lem}\label{l5.2}
Let $L$ satisfy Assumptions (A) and (B), $\oz$ satisfy Assumption
(C) and $M> \frac n2 (\frac1{p_\oz}-\frac 12)$. Suppose that $T$ is
a linear (resp. nonnegative sublinear) operator, which maps
$L^2(\cx)$ continuously into weak-$L^2(\cx)$.  If there exists a
positive constant $C$ such that for all $(\oz,M)$-atom $\az$,
\begin{equation}\label{5.6}
\int_{\cx}\oz\lf(T(\lz\az)(x)\r)\,d\mu(x)\le CV(B)\oz
\lf(\frac{|\lz|}{V(B)\ro(V(B))}\r),
\end{equation}
then $T$ extends to a bounded linear (resp. sublinear) operator from
$H_{\oz,L}(\cx)$ to $L(\oz )$; moreover, there exists a positive
constant $\wz C$ such that for all $f\in H_{\oz,L}(\cx)$,
$\|Tf\|_{L(\oz)} \le \wz C \|f\|_{H_{\oz,L}(\cx)}$.
\end{lem}

\begin{thm}\label{t5.2}
 Let $L$ satisfy Assumptions (A) and (B) and $\oz$ satisfy Assumption (C).
 Then the spaces $H_{\oz,L}(\cx)$ and $H_{\oz,\cs_P}(\cx)$
 coincide with equivalent norms.
\end{thm}

\begin{proof} Let us first prove that $H_{\oz,L}(\cx)\subset H_{\oz,\cs_P}(\cx)$.
By \eqref{2.8}, we have that the operator $\cs_P$ is bounded on
$L^2(\rn)$. Thus, by Lemma \ref{l5.2}, to show that
$H_{\oz,L}(\cx)\subset H_{\oz,\cs_P}(\cx)$, we only need to show
that \eqref{5.6} holds with $T=\cs_P$, where $M\in\cn$ and $M> \frac
n2 (\frac1{p_\oz}-\frac 12)$. Indeed, it is enough to show that for
all $f\in H_{\oz,L}(\cx)\cap L^2(\cx)$, $\|\cs_P (f)\|_{L(\oz)}\ls
\|f\|_{H_{\oz,L}(\cx)},$ which implies that $(H_{\oz,L}(\cx)\cap
L^2(\cx))\subset H_{\oz,\cs_P}(\cx)$. Then by the completeness of
$H_{\oz,L}(\cx)$ and $H_{\oz,\cs_P}(\cx)$, we see that
$H_{\oz,L}(\cx)\subset H_{\oz,\cs_P}(\cx)$.

Suppose that $\lz\in\cc$ and $\az$ is an $(\oz,M)$-atom supported in
$B\equiv B(x_B,r_B)$. Since $\oz$ is concave, by the Jensen
inequality and the H\"older inequality, we have
\begin{eqnarray*}
\int_\cx \oz(\cs_P(\lz\az)(x))\,d\mu(x)&&=\sum_{k=0}^\fz
\int_{U_k(B)} \oz(|\lz|\cs_P(\az)(x))\,d\mu(x)\\
&&\le \sum_{k=0}^\fz V(2^kB)\oz\bigg(\frac{|\lz|\int_{U_k(B)}
\cs_P(\az)(x)\,d\mu(x)}{V(2^kB)}\bigg)\\
&&\le\sum_{k=0}^\fz
V(2^kB)\oz\bigg(\frac{|\lz|\|\cs_P(\az)\|_{L^2(U_k(B))}}
{[V(2^kB)]^{1/2}}\bigg).
\end{eqnarray*}

Since $\cs_P$ is bounded on $L^2(\cx)$, for $k=0,\,1,\,2,$ we have
$$\|\cs_L(\az)\|_{L^2(U_k(B))}\ls \|\az\|_{L^2(\cx)}\ls
[V(B)]^{-1/2}[\ro(V(B))]^{-1}.$$

For $k\ge 3$, write
\begin{eqnarray*}
\|\cs_P(\az)\|_{L^2(U_k(B))}^2&&=\int_{U_k(B)}\int_0^{\frac{d(x,x_{B})}{4}}
\int_{d(x,y)<t}
|t\sqrt Le^{-t\sqrt L}\az(y)|^2\frac{\,d\mu(y)}{V(x,t)}\frac{\,dt}{t}\,d\mu(x)\\
&&\hs+\int_{U_k(B)}\int_{\frac{d(x,x_{B})}{4}}^\fz\int_{d(x,y)<t}
\cdots \equiv \mathrm{I}_k+\mathrm{J}_k.
\end{eqnarray*}

Since $\az$ is an $(\oz,M)$-atom, by Definition \ref{d4.2}, we have
$\az=L^Mb$ with $b$ as in Definition \ref{d4.2}. To estimate
$\mathrm{I}_k$, let $F_k({B}) \equiv \{y\in\cx: \
d(x,y)<d(x,x_{B})/4 \ \mathrm {for \ some }\ x\in U_k({B})\}$. Then
we have
$d(y,z)\ge d(x,x_B)-d(z,x_B)-d(y,x)\ge \frac 34 d(x,x_B)-r_B\ge 2^{k-2}r_B.$
By Lemma \ref{l5.1}, we have
\begin{eqnarray*}
\mathrm{I}_k&&\ls\int_0^{2^{k-2}r_{B}}\int_{F_k({B})} |(t\sqrt
L)^{2M+1}e^{-t\sqrt L}b(y)|^2\,d\mu(y)
\frac{\,dt}{t^{4M+1}}\\
&&\ls \|b\|_{L^2(\cx)}^2\int_0^{2^{k-2}r_{B}}
\lf(\frac{t}{\dist(F_k({B}),{B})}\r)^{4M+2} \frac{\,dt}{t^{4M+1}}
\ls 2^{-4kM}[V({B})]^{-1}[\ro(V({B}))]^{-2}.
\end{eqnarray*}

For the term $\mathrm{J}_k$, notice that if $x\in U_k({B})$, then we have
$d(x,x_{B})\ge 2^{k-1}r_{B}$, which together with Lemma \ref{l5.1} yields
that
\begin{eqnarray*}
\mathrm{J}_k&&\ls \int_{2^{k-3}r_{B}}^\fz
\int_\cx|(t\sqrt L)^{2M+1}e^{-t\sqrt L}b(y)|^2\,d\mu(y)
\frac{\,dt}{t^{4M+1}}\\
&&\ls\int_{2^{k-3}r_{B}}^\fz\|b\|_{L^2(\cx)}^2\frac{\,dt}{t^{4M+1}}
\ls 2^{-4kM}[V({B})]^{-1}[\ro(V({B}))]^{-2}.
\end{eqnarray*}

Using the estimates of $\mathrm{I}_k$ and $\mathrm{J}_k$ together
with the strictly lower type $p_\oz$ of $\oz$, we obtain
\begin{eqnarray*}
V(2^kB)\oz\bigg(\frac{|\lz|\|\cs_P(\az)\|_{L^2(U_k({B}))}}
{[V(2^kB)]^{1/2}}\bigg)&&\ls 2^{-2kMp_\oz}V(2^kB)
\lf(\frac{V({B})}{V(2^kB)}\r)^{p_\oz/2}
\oz\bigg(\frac{|\lz|}{V({B})\ro(V({B}))}\bigg)\\
&&\ls 2^{-k[2Mp_\oz-n(1-p_\oz/2)]}V({B})\oz\bigg(\frac{|\lz|}
{V({B})\ro(V({B}))}\bigg),
\end{eqnarray*}
where $2Mp_\oz>n(1-p_\oz/2)$. Thus, we finally obtain that
$$\int_\cx \oz(\cs_P(\lz\az)(x))\,d\mu(x)\ls V(B)\oz\bigg(\frac{|\lz|}
{V({B})\ro(V({B}))}\bigg),$$
that is, \eqref{5.6} holds. This finishes the proof of the
inclusion of $H_{\oz,L}(\cx)$ into $H_{\oz,\cs_P}(\cx)$.

Conversely, for any $f\in H_{\oz,\cs_P}(\cx)\cap L^2(\cx)$, we have
$t\sqrt Le^{-t\sqrt L}f\in T_{\oz}(\cx)$, which together with
Proposition \ref{p4.2}(ii) implies that $\pi_{\Psi,L}(t\sqrt
Le^{-t\sqrt L}f)\in H_{\oz,L}(\cx)$.

On the other hand, by $H_\fz$-functional calculus, we have
$f=\frac{\wz C_\Psi}{C_\Psi}\pi_{\Psi,L}(t\sqrt Le^{-t\sqrt L}f)$
in $L^2(\cx)$, where $\wz C_\Psi$ is the positive constant such that
$\wz C_\Psi\int_0^\fz \Psi(t)te^{-t}\frac{\,dt}{t}=1$ and $C_\Psi$
is as in \eqref{4.7}. This, combined with the fact
$\pi_{\Psi,L}(t\sqrt Le^{-t\sqrt L}f)\in H_{\oz,L}(\cx)$, implies
that $f\in H_{\oz,L}(\cx)$. Via a density argument, we further
obtain $H_{\oz,\cs_P}(\cx)\subset H_{\oz,L}(\cx)$, which completes
the proof of Theorem \ref{t5.2}.
\end{proof}

\begin{rem}\label{r5.1}\rm
(i) Since the atoms are associated with $L$, they do not have
vanishing integral in general. Hofmann et al \cite{hlmmy} introduced
the so-called the conservation property of the semigroup, namely,
$e^{-tL}1=1$ in $L^2_{\loc}(\cx)$, and showed that under this
assumption and Assumptions (A) and (B), then for each $(1,M)$-atom
$\az$, $\int_\cx \az(x)\,d\mu(x)=0$. From this and Proposition
\ref{p4.3}, we immediately deduce that if $L$ satisfies Assumptions
(A) and (B) and has the conservation property, and $\oz$ satisfies
Assumption (C) with $p_\oz\in (n/(n+1), 1]$, then
$H_{\oz,L}(\cx)\subset H_\oz(\cx)$, where $\cx$ is an Ahlfors
$n$-regular space (see \cite{v}). In particular,
$H_{L}^p(\cx)\subset H^p(\cx)$ for all $p\in (n/(n+1),1]$.

(ii) Let $s\in\zz_+$. The semigroup $\{e^{-tL}\}_{t>0}$ is said to
have the $s$-generalized conservation property, if for all $\gz\in \zz_+^n$
with $|\gz|\le s$,
\begin{eqnarray}\label{5.7}
e^{-tL}((\cdot)^\gz)(x)=x^{\gz}\quad \mathrm{in}\quad L^2_{\loc}(\rn),
\end{eqnarray}
namely, for every $\phi\in L^2(\rn)$
with bounded support,
\begin{eqnarray}\label{5.8}
  \int_\rn x^{\gz}e^{-tL}\phi(x)\,dx\equiv
  \int_\rn e^{-tL}((\cdot)^\gz)(x)\phi(x)\,dx=
  \int_\rn x^{\gz}\phi(x)\,dx,
\end{eqnarray}
where $x^\gz=x_1^{\gz_1}\cdots x_n^{\gz_n}$
for $x=(x_1,\cdots,x_n)\in\rn$ and $\gz=(\gz_1,\cdots,\gz_n)\in
\zz_+^n$.

Notice that for any $\phi$ with bounded support and $\gz\in
\zz_+^n$, by the Davies-Gaffney estimate, one can easily check that
$x^{\gz}e^{-tL}\phi(x),\, x^{\gz}(I+L)^{-1}\phi(x)\in L^1(\rn)$.
Hence, by \eqref{5.8} and the $L^2(\rn)$-functional calculus, we
obtain
\begin{eqnarray}\label{5.9}
\int_\rn x^{\gz}(I+L)^{-1}\phi(x)\,dx=
\int_0^\fz e^{-t}\lf[\int_\rn x^{\gz} e^{-tL}\phi(x)\,dx\r]\,dt=
\int_\rn x^{\gz}\phi(x)\,dx.
\end{eqnarray}

Let $\az$ be an $(\oz,M)$-atom and $s\equiv \lfr n(\frac 1 {p_\oz}-1) \rf$.
By Definition \ref{d4.2}, there exists
$b\in \cd(L^M)$ such that $\az=L^Mb$. Thus, if $L$ satisfies \eqref{5.7},
then for all $\gz\in \zz_+^n$
and $|\gz|\le s$, by \eqref{5.8} and \eqref{5.9}, we obtain
\begin{eqnarray*}
&&\int_\rn x^{\gz}\az(x)\,d\mu(x)\\
&&\hs=\int_\rn x^{\gz} (I+L)^{-1}\az(x)\,d\mu(x)\\
&&\hs=\int_\rn x^{\gz} (I+L)^{-1}(I+L) L^{M-1}b(x)\,d\mu(x)
-\int_\rn x^{\gz} (I+L)^{-1} L^{M-1}b(x)\,d\mu(x)\\
&&\hs=\int_\rn x^{\gz}  L^{M-1}b(x)\,d\mu(x)
-\int_\rn x^{\gz} (I+L)^{-1} L^{M-1}b(x)\,d\mu(x)=0,
\end{eqnarray*}
which implies that $\az$ is a classical $H_\oz(\rn)$-atom; for the definition
of $H_\oz(\rn)$-atoms, see \cite{v}.

Thus, if $L$ satisfies \eqref{5.7} and Assumptions (A) and (B), and
$\oz$ satisfies Assumption (C), then by Proposition \ref{p4.3}, we
know that $H_{\oz,L}(\rn)\subset H_\oz(\rn)$. In particular,
$H_{L}^p(\rn)\subset H^p(\rn)$ for all $p\in (0,1]$.
\end{rem}

\section{ Applications to Schr\"odinger operators \label{s6}}

\hskip\parindent In this section, let $\cx\equiv \rn$ and $L\equiv
-\Delta +V$ be a Schr\"odinger operator, where $V\in
L^1_{\loc}(\rn)$ is a nonnegative potential. We establish several
characterizations of the corresponding Orlicz-Hardy spaces
$H_{\oz,L}(\rn)$ by beginning with some notions.

Since $V$ is a  nonnegative function, by the Feynman-Kac formula,
we obtain that $h_t$, the kernel of the semigroup $e^{-tL}$,
satisfies that for all $x,\,y\in\rn$ and $t\in(0,\fz)$,
\begin{equation}\label{6.1}
  0\le h_t(x,y)\le (4\pi t)^{-n/2}\exp\lf(-\frac{|x-y|^2}{4t}\r).
\end{equation}
It is easy to see that $L$ satisfies Assumptions (A) and (B).

From Theorems \ref{t5.1} and \ref{t5.2}, we deduce the following
conclusions on Hardy spaces associated with $L$.

\begin{thm}\label{t6.1}
Let $\oz$ be as in Assumption (C), $M> \frac n2(\frac 1{p_\oz}-\frac
12)$ and $\ez>n(1/p_\oz-1/p_\oz^+)$. Then the spaces
$H_{\oz,L}(\rn)$, $H_{\oz,{\rm at}}^{M}(\rn)$, $H_{\oz,{\rm
mol}}^{M,\ez}(\rn)$ and $H_{\oz,\cs_P}(\rn)$ coincide with
equivalent norms.
\end{thm}

Let us now establish the boundedness of the Riesz transform $\nz L^{-1/2}$
on $H_{\oz,L}(\rn)$. We first recall a lemma established in \cite{hlmmy}.

\begin{lem}\label{l6.1}
There exist two positive constants $C$ and $c$ such that for
all closed sets $E$ and $F$ in $\rn$ and $f\in L^2(E)$,
$$\|t\nz e^{-t^2L}f\|_{L^2(F)}\le C\exp\lf\{-\frac{\dist(E,F)^2}
{ct^2}\r\}\|f\|_{L^2(E)}.$$
\end{lem}

\begin{thm}\label{t6.2}
Let $\oz$ be as in Assumption (C). Then the Riesz transform $\nz L^{-1/2}$
is bounded from $H_{\oz,L}(\rn)$ to $L(\oz)$.
\end{thm}

\begin{proof}
It was proved in \cite{hlmmy} that the Riesz transform
$\nz L^{-1/2}$ is bounded on $L^2(\rn)$; thus, to prove Theorem
\ref{t6.2}, by Lemma \ref{l5.2}, it suffices to show that
\eqref{5.6} holds.

Suppose that $\lz\in \cc$ and $\az=L^Mb$ is an $(\oz,M)$-atom
supported in $B\equiv B(x_B,r_B)$, where $b$ is as in Definition
\ref{d4.2} and we choose $M\in\cn$ such that $M>\frac n2(\frac
1{p_\oz}-\frac12)$.

For $j=0,1,2$, by the Jensen inequality, the H\"older inequality
and the $L^2(\rn)$-boundedness of $\nz L^{-1/2}$, we obtain
\begin{equation*}
\int_{U_j(B)}\oz(|\lz\nz L^{-1/2}\az(x)|)\,dx\ls
|B|\oz\lf(\frac{\|\lz\nz L^{-1/2}\az\|_{L^2(U_j(B))}}{|B|^{1/2}}\r)
\ls |B|\oz\lf(\frac{|\lz|}{\ro(|B|)|B|}\r).
\end{equation*}

Let us estimate the case $j\ge 3$. By \cite[Lemma 2.3]{hm1}, we see
that the operator $t\nz(t^2L)^M e^{-t^2L}$ also satisfies the
Davies-Gaffney estimate. By this, the fact that $\oz^{-1}$ is
convex, the Jensen inequality, the H\"older inequality and Lemma
\ref{l6.1}, we obtain
\begin{eqnarray*}
&&\oz^{-1}\lf(\frac{1}{|U_j(B)|}\int_{U_j(B)}\oz(|\lz\nz
L^{-1/2}\az(x)|)\,dx\r)\\&&
\hs\ls\frac{1}{|U_j(B)|}\int_{U_j(B)}\oz^{-1}\circ\oz\lf(\lf|\int_0^{\fz}
\lz\nz e^{-t^2L}\az(x)\,dt\r|\r)\,dx\\
&&\hs\ls \frac{1}{|U_j(B)|}\int_0^{\fz}\int_{U_j(B)} |\lz
t\nz(t^2L)^M
e^{-t^2L}b(x)|\,dx\frac{\,dt}{t^{2M+1}}\\
&&\hs\ls \frac{|\lz| \|b\|_{L^2(\rn)}}{{|U_j(B)|^{1/2}}}\int_0^{\fz}
\exp\lf\{-\frac{\dist(B,U_j(B))^2}{ct^2}\r\}\frac{\,dt}{t^{2M+1}}\\
&&\hs\ls \frac{|\lz| \|b\|_{L^2(\rn)}}{{|U_j(B)|^{1/2}}}\int_0^{\fz}
(2^jr_B)^{-2M}\min\lf\{\frac{t}{2^jr_B},\frac{2^jr_B}{t}\r\}\frac{\,dt}{t}
\ls 2^{-j(2M+n/2)}\frac{|\lz|}{\ro(|B|)|B|},
\end{eqnarray*}
where $c$ is a positive constant. Since $\oz$ is of lower type
$p_\oz$, we obtain
\begin{eqnarray*}
\int_{U_j(B)}\oz(|\lz\nz L^{-1/2}\az(x)|)\,dx&&\ls
|U_j(B)|\oz\lf(2^{-j(2M+n/2)}\frac{|\lz|}{\ro(|B|)|B|}\r)\\
&&\ls 2^{-j[2Mp_\oz+n(p_\oz/2-1)]}|B|\oz\lf(\frac{|\lz|}{\ro(|B|)|B|}\r).\nonumber
\end{eqnarray*}

Combining the above estimates and using $M>\frac n2(\frac
1{p_\oz}-\frac12)$, we obtain that \eqref{5.6} holds for $\nz
L^{-1/2}$, which completes the proof of Theorem \ref{t6.2}.
\end{proof}

It was proved in \cite{hlmmy} that the Riesz transform $\nz
L^{-1/2}$ is bounded from $H_L^1(\rn)$ to $H^1(\rn)$. Similarly to
\cite[Theorem 7.4]{jy}, we have the following result. We omit the
details here; see \cite{ja,v,se96} for more details about the
Hardy-Orlicz space $H_{\oz}(\rn)$.

\begin{thm}\label{t6.3}
Let $\oz$ be as in Assumption (C) and $p_\oz\in (\frac{n}{n+1},1]$.
Then the Riesz transform $\nz L^{-1/2}$
is bounded from $H_{\oz,L}(\rn)$ to $H_{\oz}(\rn)$.
\end{thm}

We next characterize the Orlicz-Hardy space $H_{\oz,L}(\rn)$ via
maximal functions. To this end, we first introduce some notions.

Let $\nu>0$. Recall that for all measurable function $g$ on $\rnz$ and
$x\in\rn$, the Lusin area function $\ca_\nu (g)(x)$ is defined by
$\ca_\nu(g)(x)\equiv(\int_{\Gamma_\nu(x)}|g(y,t)|^2
\frac{\,dy\,dt}{t^{n+1}})^{1/2}.$
Also the non-tangential maximal function is defined by
$\nn^\nu (g)(x)\equiv \sup_{(y,t)\in \Gamma_\nu (x)}|g(y,t)|.$

\begin{lem}\label{l6.2}
Let $\eta,\,\nu\in(0,\fz)$. Then there exists a positive constant
$C$, depending on $\eta$ and $\nu$, such that for all measurable
function $g$ on $\rr^{n+1}_+$,
\begin{equation}\label{6.2}
C^{-1}\int_{\rn} \oz(\ca_\eta (g)(x))\,dx\le \int_{\rn} \oz(\ca_\nu
(g)(x))\,dx\le C\int_{\rn} \oz(\ca_\eta (g)(x))\,dx
\end{equation}
and
\begin{equation}\label{6.3}
C^{-1}\int_{\rn} \oz(\nn^\eta (g)(x))\,dx\le \int_{\rn} \oz(\nn^\nu
(g)(x))\,dx\le C\int_{\rn} \oz(\nn^\eta (g)(x))\,dx.
\end{equation}
\end{lem}

\begin{proof} \eqref{6.2} was established in \cite[Lemma 3.2]{jy},
while \eqref{6.3} can be proved by an argument similar to those used
in the proofs of \cite[Theorem 2.3]{ct} and \cite[Lemma 5.3]{jy}. We omit the
details, which completes the proof of Lemma \ref{l6.2}.
\end{proof}

For any $\bz\in(0,\fz)$, $f\in L^2(\rn)$ and $x\in\rn$, let
$\nn_h^\bz(f)(x)\equiv \nn^\bz(e^{-t^2L}f)(x)$,
$\nn_P^\bz (f)(x)\equiv\nn^\bz(e^{-t\sqrt L}f)(x),$
$\car_h(f)(x)\equiv\sup_{t>0}|e^{-t^2 L}f(x)|$ and
$$\car_P(f)(x)\equiv\sup_{t>0}|e^{-t \sqrt L}f(x)|.$$ We denote
$\nn_h^1(f)$ and $\nn_P^1(f)$ simply by $\nn_h(f)$ and $\nn_P(f)$,
respectively.

Similarly to Definition \ref{d4.1}, we introduce the space $H_{\oz,
\nn_h}(\rn)$ as follows.

\begin{defn}
Let $\oz$ be as in Assumption (C) and $\overline{\car(L)}$ as in
\eqref{4.2}. A function $f\in \overline{\car(L)}$ is said to be in
$\wz H_{\oz,\nn_h}(\rn)$ if $\nn_h(f)\in L(\oz)$; moreover, define
$$\|f\|_{H_{\oz,\nn_h}(\rn)}\equiv \|\nn_h(f)\|_{L(\oz)}=
\inf\lf\{\lz>0:\, \int_\rn \oz\lf(\frac{\nn_h(f)(x)}{\lz}\r)\,dx\le
1\r\}.$$ The Hardy space $H_{\oz,\nn_h}(\rn)$ is defined to be the
completion of $\wz H_{\oz,\nn_h}(\rn)$ in the norm
$\|\cdot\|_{H_{\oz,\nn_h}(\rn)}.$
\end{defn}

The spaces $H_{\oz,\nn_P}(\rn)$, $H_{\oz,\car_h}(\rn)$ and
$H_{\oz,\car_P}(\rn)$ are defined in a similar way.

The following Moser type local boundedness estimate was established
in \cite{hlmmy}.

\begin{lem}\label{l6.3}
Let $u$ be a weak solution of $\wz L u\equiv Lu-\pa_{t}^2u =0$ in the ball
$B(Y_0,2r)\subset \rnz$. Then for all $p\in (0,\fz)$,
there exists a positive constant $C(n,p)$ such that
$$\sup_{Y\in B(Y_0,r)} |u(Y)|\le C(n,p)
\lf(\frac{1}{r^{n+1}}\int_{B(Y_0,2r)} |u(Y)|^p\,dY\r)^{1/p}.$$
\end{lem}

To establish the maximal function characterizations
of $H_{\oz,L}(\rn)$, an extra assumption on $\oz$ is needed.

\begin{proof}[\bf Assumption (D)]
Let $\oz$ satisfy Assumption (C). Suppose that there exist
$q_1,\,q_2\in (0,\fz)$ such that $q_1<1<q_2$ and
$[\oz(t^{q_2})]^{q_1}$ is a convex function on $(0,\fz)$.
\end{proof}

Notice that if $\oz(t)=t^p$ with $p\in (0, 1]$ for all $t\in
(0,\fz)$, then for all $q_1\in (0,1)$ and $q_2\in [1/(pq_1),\fz)$,
$[\oz(t^{q_2})]^{q_1}$ is a convex function on $(0,\fz)$; if
$\oz(t)=t^{1/2}\ln(e^4+t)$ for all $t\in (0,\fz)$, then it is easy
to check that $[\oz(t^4)]^{1/2}$ is a convex function on $(0,\fz)$.

\begin{thm}\label{t6.4}
Let $\oz$ be as in Assumption (D). Then the spaces $H_{\oz,L}(\rn)$,
$H_{\oz,\nn_h}(\rn)$, $H_{\oz,\nn_P}(\rn)$, $H_{\oz,\car_h}(\rn)$ and
$H_{\oz,\car_P}(\rn)$ coincide with equivalent norms.
\end{thm}

\begin{proof}
We first show that $H_{\oz,L}(\rn)\subset H_{\oz,\nn_h}(\rn)$.
By \eqref{6.1}, for all $f\in L^2(\rn)$ and $x\in\rn$, we have
\begin{eqnarray*}
\nn_h(f)(x)&&\ls \sup_{y\in B(x,t)}t^{-n}\int_{\rn}
\exp\lf(-\frac{|y-z|^2}{4t^2}\r)|f(z)|\,dz\\
&&\ls \sum_{j=0}^\fz\sup_{y\in B(x,t)}t^{-n}\int_{U_j(B(y,2t))}
\exp\lf(-\frac{|y-z|^2}{4t^2}\r)|f(z)|\,dz\\
&&\ls \cm(f)(x)+ \sum_{j=2}^\fz\sup_{y\in B(x,t)}t^{-n}
2^{-j(n+1)}\int_{U_j(B(y,2t))}  |f(z)|\,dz\ls \cm(f)(x),
\end{eqnarray*}
where $\cm$ is the Hardy-Littlewood maximal function on $\rn$. Thus, $\nn_h$
is bounded on $L^2(\rn)$.

Thus, by Lemma \ref{l5.2} and the completeness of $H_{\oz,L}(\rn)$
and $H_{\oz,\nn_h}(\rn)$, similarly to the proof of Theorem
\ref{t5.2}, we only need to show that for each $(\oz,M)$-atom $\az$,
\eqref{5.6} holds with $T=\nn_h$, where $M\in\cn$ and $M>\frac
n2(\frac 1{p_\oz}-\frac12)$.

To this end, suppose that $\az$ is an $(\oz,M)$-atom and
$\supp \az\subset B\equiv B(x_B,r_B)$. For
$j=0,\,\cdots,\,10$, since $\nn_h$ is bounded on $L^2(\rn)$,
by the Jensen inequality and the H\"older inequality, we have
that for any $\lz\in\cc$,
$$\int_{U_j(B)}\oz\lf(\nn_h(\lz\az)(x)\r)\,dx\ls
|U_j(B)|\oz\lf(\frac{\|\lz\nn_h(\az)\|_{L^2(U_j(B))}}{|B|^{1/2}}\r)\ls
|B|\oz\lf(\frac{|\lz|}{\ro(|B|)|B|}\r).$$

For $j\ge 11$ and $x\in U_j(B)$, let $a\in (0,1)$ such that
$ap_\oz(2M+n)>n$. Write
\begin{eqnarray*}
\nn_h (\az)(x)&&\le \sup_{y\in B(x,t),\,t\le 2^{aj-2}r_B} |e^{-t^2L}(\az)(y)|
+\sup_{y\in B(x,t),\,t> 2^{aj-2}r_B} |e^{-t^2L}(\az)(y)|
\equiv \mathrm{H}_j+\mathrm{I}_j.
\end{eqnarray*}
To estimate $\mathrm{H}_j$, observe that if $x\in U_j(B)$, then we
have $|x-x_B|> 2^{j-1}r_B$, and if $z\in B$ and $y\in F_j(B)\equiv
\{y\in\cx: \ |x-y|<2^{aj-2}r_B \ \mathrm {for \ some }\ x\in
U_j(B)\}$, then we have
$$|y-z|\ge |x-x_B|-|z-x_B|-|y-x|\ge 2^{j-1}r_B-r_B-2^{aj-2}r_B\ge 2^{j-3}r_B.$$
By \eqref{6.1}, we obtain
\begin{eqnarray*}
\mathrm{H}_j&&\ls\sup_{y\in B(x,t),\,t\le 2^{aj-2}r_B}
\frac{1}{t^n}\int_B e^{-\frac{|z-y|^2}{4t^2}}|\az(z)|\,dz\\
&&\ls \sup_{t\le 2^{aj-2}r_B}\frac{1}{t^n}
\lf(\frac{t}{2^{j}r_B}\r)^{N+n} \|\az\|_{L^1(B)}
\ls 2^{-j[n+(1-a)N]}|B|^{-1}[\ro(|B|)]^{-1},
\end{eqnarray*}
where $N\in \cn$ satisfies that $p_\oz[n+(1-a)N]>n$.

For the term $\mathrm{I}_j$, notice that since the kernel $h_t$ of
$\{e^{-t^2L}\}_{t>0}$ satisfies \eqref{6.1}, we have that for each
$k\in\cn$, there exist two positive constants $c_k$ and $\wz c_k$
such that for almost every $x,\,y\in\rn$,
\begin{equation}\label{6.4}
\lf|\frac{\pa^k}{\pa t^k}h_t(x,y)\r|\le \frac {\wz c_k}{t^{k+n/2}}
\exp\lf\{-\frac{|x-y|^2}{c_kt}\r\};
\end{equation}
see \cite{da,hlmmy}. On the other hand, since $\az$ is an $(\oz,M)$-atom,
by Definition \ref{d4.2}, we have $\az=L^Mb$ with $b$ as in Definition
\ref{d4.2}, which together with \eqref{6.4} implies that
\begin{eqnarray*}
\mathrm{I}_j&&=\sup_{y\in B(x,t),\,t> 2^{aj-2}r_B}t^{-2M} |(t^2L)^Me^{-t^2L}(b)(y)|\\
&&\ls \sup_{y\in B(x,t),\,t> 2^{aj-2}r_B}t^{-2M-n}
\int_B e^{-\frac{|z-y|^2}{c_Mt^2}}|b(z)|\,dz\ls 2^{-aj(2M+n)}|B|^{-1}[\ro(|B|)]^{-1}.
\end{eqnarray*}

Combining the above two estimates, we obtain
\begin{eqnarray*}
&&\sum_{j=11}^\fz\int_{U_j(B)}\oz(\nn_h(\lz\az)(x))\,dx\\
&&\hs\ls \sum_{j=11}^\fz|U_j(B)| \lf[2^{-jp_\oz[n+(1-a)N]}+2^{-jap_\oz[n+2M]}\r]
\oz\lf(\frac{|\lz|}{\ro(|B|)|B|}\r)\ls |B|\oz\lf(\frac{|\lz|}{\ro(|B|)|B|}\r).
\end{eqnarray*}
Thus, \eqref{5.6} holds with $T=\nn_h$, and hence
$H_{\oz,L}(\rn)\subset H_{\oz,\nn_h}(\rn)$.

From the fact that for all $f\in L^2(\rn)$, $\car_h(f)\le \nn_h(f)$,
it follows that for all $f\in H_{\oz,\nn_h}(\rn)\cap L^2(\rn)$,
$\|f\|_{H_{\oz,\car_h}(\rn)}\le \|f\|_{H_{\oz,\nn_h}(\rn)}$, which
together with a density argument implies that
$H_{\oz,\nn_h}(\rn) \subset H_{\oz,\car_h}(\rn)$.

To show that $H_{\oz,\car_h}(\rn)\subset H_{\oz,\car_P}(\rn)$, by
\eqref{5.5}, we have that for all $f\in L^2(\rn)$ and $x\in \rn$,
\begin{eqnarray*}
\car_P(f)(x)&&=\sup_{t>0}|e^{-t\sqrt L}f(x)|\ls \sup_{t>0}\int_0^\fz
\frac{e^{-u}}{\sqrt u} |e^{-\frac {t^2}{4u}L}f(x)|\,du\\
&&\ls \car_h(f)(x)\int_0^\fz \frac{e^{-u}}{\sqrt u}\,du\ls \car_h(f)(x),
\end{eqnarray*}
which implies that for all $f\in H_{\oz,\car_h}(\rn)\cap L^2(\rn)$,
$\|f\|_{H_{\oz,\car_P}(\rn)}\ls \|f\|_{H_{\oz,\car_h}(\rn)}$. Then
by a density argument, we obtain $H_{\oz,\car_h}(\rn) \subset
H_{\oz,\car_P}(\rn)$.

Let us now show that $H_{\oz,\car_P}(\rn)\subset
H_{\oz,\nn_P}(\rn)$. Since $\oz$ satisfies Assumption (D), there
exist $q_1,\,q_2\in (0,\fz)$ such that $q_1<1<q_2$ and
$[\oz(t^{q_2})]^{q_1}$ is a convex function on $(0,\fz)$.

For all $x\in \rn$, $t\in (0,\fz)$ and $f\in L^2(\rn)$, let
$u(x,t)\equiv e^{-t\sqrt L}f(x)$. Then $\wz Lu= Lu-\pa_{t}^2u=0$.
Applying Lemma \ref{l6.3} to such a $u$ with $1/q_2$, we obtain that
for all $y\in B(x,t/4)$,
\begin{eqnarray*}
|e^{-t\sqrt L}f(y)|^{1/q_2}&&\ls  \frac{1}{t^{n+1}}
\int_{t/2}^{3t/2}\int_{B(x,t/2)} |e^{-s\sqrt L}f(z)|^{1/q_2}\,dz\,ds
\ls \frac{1}{t^{n}}\int_{B(x,t)} [\car_P (f)(z)]^{1/q_2}\,dz.
\end{eqnarray*}
Since $[\oz(t^{q_2})]^{q_1}$ is convex on $(0,\fz)$, by the Jensen
inequality, we obtain
\begin{eqnarray*}
\lf[\oz(|e^{-t\sqrt L}f(y)|)\r]^{q_1}&&\ls\lf[\oz\lf(\lf(\frac{1}{t^{n}}\int_{B(x,t)}
[\car_P (f)(z)]^{1/q_2}\,dz\r)^{q_2}\r)\r]^{q_1}\\
&&\ls\frac{1}{t^{n}}\int_{B(x,t)}
\lf[\oz(\car_P (f)(z))\r]^{q_1}\,dz\ls \cm\lf([\oz(\car_P (f))]^{q_1}\r)(x),
\end{eqnarray*}
which together with the fact that $\oz$ is continuous
implies that for all $x\in\rn$,
$$\oz\lf(\nn_P^{1/4}(f)(x)\r)\ls
\lf[\cm\lf([\oz(\car_P (f))]^{q_1}\r)(x)\r]^{1/q_1}.$$

Now by \eqref{6.3} and the fact that $\cm$ is bounded on
$L^{1/q_1}(\rn)$, we obtain
\begin{eqnarray*}
\|\oz(\nn_P(f))\|_{L^1(\rn)}&&\ls\|\oz(\nn_P^{1/4}(f))\|_{L^1(\rn)}\\
&&\ls \|\lf[\cm\lf([\oz(\car_P (f))]^{q_1}\r)\r]^{1/q_1}\|_{L^1(\rn)}
\ls\|\oz(\car_P(f))\|_{L^1(\rn)},
\end{eqnarray*}
and hence $\|f\|_{H_{\oz, \nn_P}(\rn)}\ls \|f\|_{H_{\oz,
\car_P}(\rn)}$. Then by a density argument, we obtain that
$H_{\oz,\car_P}(\rn)\subset H_{\oz,\nn_P}(\rn)$.

Finally, let us show that
$H_{\oz,\nn_P}(\rn)\subset H_{\oz,L}(\rn)$.
For all $x\in \rn$, $\bz\in (0,\fz)$ and $f\in L^2(\rn)$, define
$\wz \cs_P^{\bz}f(x)\equiv(\iint_{\Gamma_\bz(x)}|t
\wz \nz e^{-t\sqrt L}f(y)|^2\frac{\,dy\,dt}{t^{n+1}})^{1/2},$
where $\wz \nz \equiv (\nz,\,\pa_t)$ and $|\wz
\nz|^2=|\nz|^2+(\pa_t)^2$. It is easy to see that $\cs_Pf\le \wz
\cs_P^{1}f$.

It was proved in the proof of \cite[Theorem 8.2]{hlmmy} that for all
$f\in L^2(\rn)$ and $u>0$,
\begin{equation}\label{6.5}
\sz_{\wz \cs_P^{1/2}f}(u)\ls
\frac{1}{u^2}\int_0^u t\sz_{\nn_P^\bz}(t)\,dt+\sz_{\nn_P^\bz}(u),
\end{equation}
where $\bz\in (0,\fz)$ is large enough,
 and $\sz_g$ denotes the distribution of the function $g$.

Since $\oz$ is of upper type 1 and lower type $p_\oz\in (0,1]$, we have
$\oz(t)\sim \int_0^t\frac{\oz(u)}{u}\,du$
for each $t\in (0,\fz)$, which together with \eqref{6.2},
\eqref{6.3}, \eqref{6.5} and $\cs_Pf\le \wz \cs_P^{1}f$, further
implies that
\begin{eqnarray*}
\int_{\rn}\oz(\cs_P(f)(x))\,dx&&\ls \int_{\rn}\oz(\wz\cs^{1}_P(f)(x))\,dx
\ls \int_{\rn}\oz(\wz\cs^{1/2}_P(f)(x))\,dx\\
&&\sim \int_\rn \int_0^{\wz\cs^{1/2}_P(f)(x)}\frac{\oz(t)}{t}\,dt\,dx
\sim \int_0^\fz \sz_{\wz\cs^{1/2}_P(f)}(t)\frac{\oz(t)}{t}\,dt\\
&& \ls \int_0^\fz \frac{\oz(t)}{t}
\lf[\frac{1}{t^2}\int_0^t u\sz_{\nn_P^\bz}(u)\,du+\sz_{\nn_P^\bz}(t)\r]\,dt\\
&&\ls \int_0^\fz  u\sz_{\nn_P^\bz}(u) \int_u^\fz\frac{\oz(t)}{t^3}\,dt\,du+
\int_\rn\oz(\nn_P^\bz(x))\,dx\\
&&\ls \int_\rn\oz(\nn_P^\bz(x))\,dx\ls \int_\rn\oz(\nn_P(x))\,dx.
\end{eqnarray*}
Thus, we obtain that $\|f\|_{H_{\oz,\cs_P}(\rn)}\ls \|f\|_{H_{\oz,\nn_P}(\rn)}$.
By Theorem \ref{t6.1}, we finally obtain that
$H_{\oz,\nn_P}(\rn)\subset H_{\oz,\cs_P}(\rn)=H_{\oz,L}(\rn)$,
which completes the proof of Theorem \ref{t6.4}.
\end{proof}

\begin{rem}\rm
(i) If $n=1$ and $p=1$, the Hardy space $H_L^1(\rn)$ coincides with
the Hardy space introduced by Czaja and Zienkiewicz in \cite{cz}.

(ii) If $L=-\Delta+V$ and $V$ belongs to the reverse H\"older class
$\ch_q(\rn)$ for some $q\ge n/2$ with $n\ge 3$, then the Hardy space
$H_L^p(\rn)$ when $p\in (n/(n+1),1]$ coincides with the Hardy space
introduced by Dziuba\'nski and Zienkiewicz \cite{dz1,dz2}.
\end{rem}

{\bf Acknowledgements.} Dachun Yang would like to thank Professor
Lixin Yan and Professor Pascal Auscher for some helpful discussions
on the subject of this paper. The authors sincerely wish to express their
deeply thanks to the referee for her/his very carefully reading and
also her/his so many careful, valuable and suggestive remarks which
essentially improve the presentation of this article.

\medskip

\noindent Renjin Jiang:

\noindent School of Mathematical Sciences, Beijing Normal University

\noindent Laboratory of Mathematics and Complex Systems, Ministry
of Education

\noindent 100875 Beijing

\noindent People's Republic of China

\noindent E-mail: \texttt{rj-jiang@mail.bnu.edu.cn}

\medskip

\noindent {\it Present address}

\noindent Department of Mathematics and Statistics,
University of Jyv\"{a}skyl\"{a}

\noindent P.O. Box 35 (MaD), 40014, Finland

\medskip
\noindent Dachun Yang (corresponding author)

\noindent School of Mathematical Sciences, Beijing Normal University

\noindent Laboratory of Mathematics and Complex Systems, Ministry
of Education

\noindent 100875 Beijing

\noindent People's Republic of China

\noindent E-mail: \texttt{dcyang@bnu.edu.cn}

\end{document}